\documentclass[12pt,a4paper,reqno]{amsart}
\usepackage[utf8]{inputenc}
\usepackage[english]{babel}
\usepackage[T1]{fontenc}
\usepackage{amsmath}
\usepackage{amsfonts}
\usepackage{amsthm}
\usepackage{amssymb}
\usepackage{graphicx}
\usepackage{lmodern}
\usepackage{ bbold }
\usepackage{leftindex}
\usepackage{caption}
\usepackage{comment}
\usepackage{hyperref}
\usepackage{titling}
\usepackage{authblk}
\usepackage{mathabx}
\usepackage{color}
\usepackage{esint}
\usepackage{calrsfs}
\usepackage{mathrsfs}
\usepackage{stmaryrd}
\usepackage{amsaddr}
\usepackage{array}
\usepackage[misc]{ifsym}
\usepackage[pagewise]{lineno}
\usepackage[toc,page]{appendix}
\usepackage[left=3cm,right=3cm,top=3cm,bottom=3cm]{geometry}

\newtheorem{theoreme}{Theorem}
\newtheorem{lemme}{Lemma}
\newtheorem{proposition}{Proposition}

\newtheorem{definition}{Definition}

\newcommand{\bea}{\begin{eqnarray}}
\newcommand{\eea}{\end{eqnarray}}
\newcommand{\beas}{\begin{eqnarray*}}
\newcommand{\eeas}{\end{eqnarray*}}

\def\[{\left[}
\def\]{\right]}
\def\<{\langle}
\def\>{\rangle}
\def\({\left(}
\def\){\right)}

\def\dd{\textnormal{d}}

\def\o{\otimes}
\def\N{{\mathbb N}}
\def\R{{\mathbb R}}

\def\Z{{\mathbb Z}}
\def\T{{\mathbb T}}

\numberwithin{equation}{section}
\numberwithin{theoreme}{section}
\numberwithin{lemme}{section}

\numberwithin{proposition}{section}
\numberwithin{algo}{section}
\numberwithin{definition}{section}
\numberwithin{remark}{section}

\title{Convergence of a Multi-fluid scheme for the Vlasov--Poisson system}

\date{}

\author{Mehdi Badsi$^1$, Nicolas Crouseilles$^2$ \and Daniel Han-Kwan$^3$}

\thanks{$^1$Nantes Université, Inria centre de l'université de Rennes (Mingus team) and Laboratoire de Mathématiques Jean Leray UMR CNRS 6310, 2 Chemin de la Houssinière BP 92208, 44322 Nantes Cedex 3, France. (mehdi.badsi@univ-nantes.fr) }
\thanks{$^2$Univ Rennes, Inria centre de l'universit\'e de Rennes (Mingus team) and IRMAR UMR CNRS 6625, F-35042 Rennes, France \& ENS Rennes, France. (nicolas.crouseilles@inria.fr)}
\thanks{$^3$Nantes Universit\'e, CNRS,  Laboratoire de Math\'ematiques Jean Leray UMR CNRS 6310, 2 Chemin de la Houssini\`ere BP 92208, 44322 Nantes Cedex 3, France. 
  (daniel.han-kwan@univ-nantes.fr).}


\begin{document}
\maketitle

\begin{abstract}
In this paper, we introduce a novel multi-fluid approximation scheme for the one-dimensional Vlasov--Poisson system, that is relevant for arbitrarily long time intervals.
For low regularity solutions, we establish an error estimate in the Wasserstein-$1$ distance, obtaining 
a convergence rate of order $\mathcal{O}(h^{\frac{2}{3}})$, where $h$ denotes the velocity grid size. This establishes the consistency, as the grid size tends to zero, between the pressureless Euler--Poisson  and the Vlasov--Poisson systems. Numerical illustrations are given to illustrate the efficiency of our approach.
\end{abstract}

\tableofcontents
\section{Introduction}

\subsection{General context}
One of the most widely used and efficient models for describing non-collisional plasmas is the Vlasov--Poisson system. It governs the evolution of the distribution function of a particle species in phase space, whose characteristic velocities are assumed to be much smaller than the speed of light. Despite its widespread use, the accurate and stable numerical simulation of the Vlasov--Poisson system remains a challenging problem. Consequently, it has received considerable attention over the past decades, with numerous works addressing a wide range of issues, including the preservation of physical conservation laws, error analysis, and model reduction techniques aimed at improving computational efficiency  \cite{cheng, christlieb_va, filbet, sonnendrucker, filbet,gamba_dg, CMS-jcp, lapenta_sl_va, francis-bessemoulin, francis-bessemoulin2,besse-bag,crestetto-bag,besse-michel,CamposPinto,gempic}. Regarding accuracy, a major difficulty in the numerical approximation of the Vlasov--Poisson system stems from the increasingly oscillatory behavior of the distribution function in the velocity variable. This phenomenon is reflected in the growth of its velocity derivatives \cite{Ukai-Okabe}. Inherent to the free transport equation, this progressive development of increasingly fine-scale structures is commonly referred to as \textit{filamentation}. This difficulty partly explains the loss of accuracy observed in Particle-in-Cell methods, which is worsened by the intrinsic statistical noise introduced during the initialization of the method \cite{MonteCarlo-method,cottet-raviart}. Semi-Lagrangian schemes and their variants have emerged as an effective compromise between accuracy and computational cost \cite{sonnendrucker,besse-michel,cdm,campos-charles-ltp}, making them a method of choice for the numerical simulation of the Vlasov--Poisson system. More recently, increasing attention has been devoted to numerical methods based on moment representations, such as Hermite-based decompositions \cite{Despres,francis-bessemoulin}. These approaches are particularly appealing from a practical point of view, as they offer a reduced-dimensional representation of the problem. The convergence of these methods is a difficult problem, the only convergence result we are aware of is \cite{francis-bessemoulin}.

\subsection{The multi-fluid approximation scheme}
The approach developed in this paper is different and is inspired from ideas that date back to the early 90's by Zakharov, Brenier, Grenier \cite{Zak, Brenier, Grenier1999} and is based on the comprehensive theory recently developed by Baradat, Ertzbischoff and Han-Kwan \cite{BEH,Bardat20} which uses an exact representation of the distribution function which leads to a point of view that is in between the Lagrangian and Eulerian formulation of the Vlasov equation, the so-called {\it multiphasic formulation}. A seemingly related, yet fundamentally different, class of more heuristic approaches consists of the so-called Multi Water-Bag and Multi-Stream methods \cite{10.1093/mnras/stu2308,besse-bag,crestetto-bag,Ghizoo}. These methods approximate the distribution function through an ansatz based on a superposition of level sets of the distribution function. Although these approaches are natural and closely related in spirit to the \textit{multi-fluid} approximation considered here, they cease to be applicable once a level set can no longer be represented as the graph of a single-valued function. In particular, this restriction prevents their use in the long-time asymptotic regime and in configurations where distinct graphs intersect.

The method we propose is based on exact solutions of the Vlasov--Poisson system that can be represented  as a discrete sum of monokinetic solutions, referred to as \textit{multi-fluid} solutions. Such solutions are a particular case of the general multiphasic framework of \cite{BEH}. In contrast to the approaches described above, for fairly general bounded and integrable initial data, with a mild decay assumption, we provide an explicit construction of a consistent \textit{multi-fluid} approximation based on a uniform velocity grid of size $h$. 
This construction leads to the study of a coupled pressureless Euler--Poisson system whose initial fluid velocities are constant. A fundamental inherent issue comes from the fact that the solutions of this system are in general bound to blow up in finite time, due to the formation of shocks for the Euler--Poisson dynamics (see e.g. \cite{ELT,CCTT,CKKT} and references therein).

In this work, we develop a numerical strategy to overcome this curse, extend the approximation beyond blow-up, and ultimately reach arbitrary long times. To this end, we introduce what we call the \textit{remapping} strategy: the idea consists, before the formation of a shock, in ``projecting'' the solution onto the uniform velocity grid, through a suitable \textit{remapping} operator, and restart the dynamics from the new resulting initial data. The remapping operator is designed so that the charge density is preserved, and is stable with respect to the Wasserstein-1 distance.

The whole numerical strategy is in principle implementable in any dimension. However, we need to restrict to the one-dimensional Vlasov--Poisson system, in order to obtain a proof of consistency; the extension to higher dimension remains open. 
In the one-dimensional case, we are able to prove that the remapping procedure can be iterated over arbitrarily long time intervals, not directly for the Vlasov--Poisson (VP)  system, but for an auxiliary regularized version, whose solution is shown to be close to the solution of the VP system. 
Our analysis finally establishes convergence of the resulting \textit{multi-fluid} scheme as the grid size tends to zero, under low regularity assumptions on the initial data. 
The proof relies on a weak--strong stability estimate in the Wasserstein-1 distance due to Hauray \cite{Hauray}, together with an $L^\infty$ estimate on the electric field. Both estimates are specific to the one-dimensional setting. The overall strategy of the proof consists of controlling two distinct sources of error: 
\begin{itemize}
    \item (Error between VP and regularized VP) The first is the consistency error between the Vlasov--Poisson system and the regularized Vlasov--Poisson system, in which the Coulomb kernel is suitably mollified near the origin. 
    \item (Error between multi-fluid regularized VP and regularized VP) The second is the approximation error between the regularized Vlasov--Poisson system and its \textit{multi-fluid} approximation.
\end{itemize}
To estimate the latter, which is the main point of this work, we propagate the error introduced at the initial time using the stability estimate up to the first remapping time. The remapping procedure then introduces an additional error of order $h$. Starting from the resulting approximation, we restart the dynamics and propagate this new error up to the next remapping time. Iterating this procedure yields, by induction, a global error estimate over arbitrary time intervals. Optimizing the regularization parameter leads to the convergence rate $\mathcal{O}(h^{2/3})$.
We emphasize that this rate holds for weak solutions with no regularity assumption at all; in particular, large gradients in $x$ or $v$ are allowed. 

\subsection{Organization of the paper}
The paper is organized as follows. Section~\ref{sec:VP} is a quick introduction to the one-dimensional Vlasov--Poisson system and its regularized version.
In Section \ref{sec:multi-fluid_solutions}, we introduce the class of \textit{multi-fluid} solutions to the Vlasov--Poisson system and recall the appropriate functional framework for establishing their existence. In particular, we derive a lower bound on the lifespan of these solutions under the assumption that the blow-up functional remains bounded by a prescribed parameter.
In Section \ref{sec:wasserstein-stab}, we recall the weak-strong stability estimate we need to propagate the errors. In Section \ref{sec:def_multi-fluid_approx_main_result}, we define the \textit{multi-fluid} approximation and state our main convergence result. The proof of the main result is in Section \ref{sec:proof_main_result}. Eventually in Section \ref{sec:numerical_results}, we use our approach to simulate the two-stream instability which is known to be challenging for this class of methods \cite{crestetto-bag}.

\section{The Vlasov--Poisson system}
\label{sec:VP}
Our main concern in this work is the approximation of the solution of the Vlasov--Poisson system posed on $[0,T] \times \T \times \R:$
\begin{equation} \label{Vlasov--Poisson_eq}
\left|
\begin{array}{cll}
& \partial_{t} f + v \partial_{x} f + E \partial_{v} f = 0, \\
&  E(t,\cdot) = K \star \Big( \displaystyle \int_{\R} f(t,\cdot,v) \dd v -1 \Big), \\
\end{array}
\right. 
\end{equation}
where $\T = \R / L \Z$ is the periodic torus of size $L > 0$, $T > 0$ is a fixed horizon of time,
\begin{itemize}
\item $K \in L^{\infty}(\T)$ is the Coulomb kernel given on $[-\frac{L}{2},\frac{L}{2})$ by
\begin{align} \label{coulomb_kernel}
    K(x) = \frac{1}{L} \begin{cases}
    -\frac{L}{2} -x & \textnormal{ if } x \in [-\frac{L}{2},0), \\
    +\frac{L}{2} - x & \textnormal{ if } x \in (0, \frac{L}{2}).
    \end{cases}
\end{align}
    \item $f(t,x,v) \geq 0$ is the distribution function of a species of ions (with positive charge) in the phase space $ \T \times \R$ at time $t \in [0,T]$, and
    \item $E(t,x) \in \R$ is the self-consistent electric field induced by the whole distribution of charges, the density of ions $\rho_{f}(t,x) = \int_{\R} f(t,x,\dd v)  $ and a fixed neutralizing background of electrons of density $1$.
\end{itemize}
The system is supplemented with an initial condition 
\begin{equation} \label{f_0}
f(0,\cdot,\cdot) = f_{0},
\end{equation}
where $f_{0}$ is assumed (without loss of generality) to be a probability measure on $\T \times \R$
\begin{align}
\label{proba}
    \frac{1}{| \T |} \int_{\T \times \R}   f_{0}( \dd x, \dd v) = 1.
\end{align}
The set of equations \eqref{Vlasov--Poisson_eq}-\eqref{proba} will be denoted (VP). In our analysis, we also consider an approximate Vlasov–Poisson system in which the Coulomb kernel \eqref{coulomb_kernel} is regularized in a neighborhood of the origin. More precisely, we introduce the regularized kernel $K_{\delta}$, defined as the $L$-periodic extension to $\mathbb{R}$ of the function given on $\left[-\frac{L}{2},\frac{L}{2}\right)$ by
\begin{align}
    K_{\delta}(x) =  -\frac{x}{L} + \begin{cases}
    -\frac{1}{2} & \textnormal{ if } x \in [-\frac{L}{2}, -\frac{\delta}{2}], \\
    \: x/\delta & \textnormal{ if } x \in (-\frac{\delta}{2}, \frac{\delta}{2}),\\
    +\frac{1}{2} & \textnormal{ if } x \in [\frac{\delta}{2}, \frac{L}{2}), 
    \end{cases}
\end{align}
where $0 < \delta < L$ is a fixed numerical parameter.
It belongs to $W^{1,\infty}(\T)$ and satisfies the bounds
\begin{align}
    & \| K_{\delta} \|_{L^{\infty}(\T)} = \frac{1}{2}\quad \textnormal{ and }\quad  \| K'_{\delta} \|_{L^{\infty}(\T)} = \frac{1}{\delta}. \label{bound_Lipschitz_K}
\end{align}
Moreover, for any $1 \leq p < +\infty$, there holds
\begin{align}
   & K_{\delta} \to K \textnormal{ in } L^{p}(\T) \quad  \textnormal{ as } \quad \delta \to 0^{+}.
\end{align}
We denote by $\textnormal{(VP)}_{\delta}$ the set of equations \eqref{Vlasov--Poisson_eq}-\eqref{proba} when the Coulomb kernel $K$ is replaced by the regularized kernel $K_{\delta}$ (with the convention that $\textnormal{(VP)}_{\delta}$ is (VP) when $\delta = 0$). The well-posedness of the Vlasov--Poisson equations is by now well-established. The existence of global weak solutions of finite energy is a classical result of Arsenev in \cite{Arsenev}. In the whole space (that is when $x \in \R^d$), the global existence and uniqueness  of classical solutions in dimension $1$ and $2$ is due to Ukai-Okabe \cite{Ukai-Okabe}; in dimension $3$ it is due to  Pfaffelmoser \cite{Pfaffelmoser} (see also \cite{Schaeffer}) and Lions-Perthame \cite{Lions-Perthame}; the latter uses an approach based on the propagation of moments in velocity.  These results also have analogues in the periodic setting (that is $x \in \T^d$), see \cite{Batt-Rein,Pallard2012}.
In  \cite{LOEPER200668}, Loeper proved that weak solutions (that may be measure valued in velocity) with bounded macroscopic density $\rho_f(t,x)$ are unique within this class (see \cite{Miot,Holding-Miot,CISS} for improvements of this criterion).
A particular well-known class of measure valued solutions corresponds to monokinetic distributions under the form $f(t,x,v)=\rho(t,x) \otimes \delta_{v= u(t,x)}$, which relate the Vlasov--Poisson equation to the pressureless Euler--Poisson system
\begin{equation} \label{EP}
\left|
\begin{array}{cc}
&  \partial_{t} \rho + \operatorname{div}_{x}(\rho u) = 0,\\
&  \partial_{t}  u + u \cdot \nabla_x u = E,\\
& E = \nabla \Delta^{-1} (\rho -1).
\end{array}
\right.
\end{equation}
In dimension $1$, the situation is actually more favorable (as the Green function associated with the Laplacian is not as singular as in higher dimension) and the Vlasov--Poisson equation actually enjoys a weak-strong uniqueness principle \cite{Hauray}, which we shall explain in more details later on.  Our analysis is based on a class of solutions with low regularity, in the sense that the proof only requires the following regularity
\begin{align}
    &f \in \mathscr{C}\big( [0,T]; L^{1}(\T \times \R)\big), \quad f \geq 0, \label{L1_continuity_in_time_L1}
\end{align}
with, for almost every $t \in [0,T],$
\begin{align}
     f(t,\cdot,\cdot) \in L^{1}(\R;L^{\infty}(\T)).
\end{align}
This integrability property follows from suitable assumptions on the initial datum. More precisely, we require the initial data
$$f_0 : \T \times \R \longrightarrow \R^{+}$$ to be measurable and to satisfy
\begin{align}
&f_0(x,v)\leq g(|v|) \quad \textnormal{ for a.e. } (x,v)\in\T\times\R,
\end{align}
for some non-increasing function $g:\R^+\to\R^+$ belonging to $L^1(\R^+; (1+|v|) \dd v)$. \begin{definition}[Admissible initial data] \label{admissible_initial_data} We denote $\mathcal{A}$ the set of initial data satisfying these assumptions, together with the normalization condition \eqref{proba}.
\end{definition}
For any $f_0\in\mathcal{A}$, Theorem 7.4 in \cite{bostan-vp-1d} ensures the existence and uniqueness of the mild solution to the Vlasov--Poisson system with the regularity
\begin{align}
(f,E)\in L^{1}([0,T]\times\T\times\R)
\times L^{\infty}([0,T];W^{1,\infty}(\T)).
\end{align}
An analysis of the mild solution also yields \eqref{L1_continuity_in_time_L1} (see the appendix for a proof of this property). Consequently, the characteristics given by
\begin{equation}
\label{eq:charac}
\left|
\begin{array}{cll}
   & X(t,s,x,v) = x + \displaystyle \int_{s}^{t} V(\tau,s,x,v) \dd \tau, \quad t\in [0,T] \\
    & V(t,s,x,v) = v + \displaystyle \int_{s}^{t} E(\tau,X(\tau,s,x,v)) \dd \tau, \quad t \in [0,T]\\
\end{array}
\right.
\end{equation}
where $(s,x,v) \in [0,T] \times \T \times \R$ is the initial condition, are uniquely defined by the classical Cauchy-Lipschitz theorem. The associated flow defined for every $s,t \in [0,T]$
\begin{align} \label{eq:nonlinear_flow_VP}
\varphi_{t,s} : (x,v) \in \T\times \R \longmapsto (X(t,s,x,v),V(t,s,x,v))
\end{align}
is a measure preserving homeomorphism. In its Lagrangian form, the solution $f(t)$ at time $t \in [0,T]$ is the transport of $f_{0}$ by the flow. If $f_{0}$ is viewed as a measure on $\T \times \R$ the transport by the characteristic flow  corresponds to the push-forward of $f_{0}$
\begin{align} \label{formula_f_mesure}
 f(t) = \varphi_{t,0}  \# f_{0}, \quad f(t)(A) = f_{0}(\varphi_{0,t}(A)) \textnormal{ for all Borel set } A \subset \T \times \R.
\end{align}
It is equivalent to test the push-forward measure $f(t)$ against any bounded and continuous function $\xi \in \mathscr{C}_{b}(\T \times \R)$
\begin{align}
\int_{\T \times \R} \xi(x,v)  f(t, \dd x, \dd v) = \int_{\T \times \R} \xi(\varphi_{t,0}(y,w))   f_0( \dd y, \dd w).
\end{align}
Finally, let us recall that the Vlasov--Poisson system \eqref{Vlasov--Poisson_eq} has many conserved quantities. When $f$ is at least measurable, any Casimir of $f$ is conserved in the sense that
\begin{equation} \label{mass_preserved}
  \forall t \in [0,T], \quad   \int_{\T \times \R} \mathrm{C}\big(f(t,x,v) \big)\dd x \dd v = \int_{\T \times \R} \mathrm{C}(f_{0}(y,w)\big) \dd y \dd w,
\end{equation}
for any $\mathscr{C}^{1}$ function $\mathrm{C} : \R \longrightarrow \R^{+}$ such that $\mathrm{C}(0) = 0.$ The case $\mathrm{C} = \textnormal{Id}$ yields the conservation of the total charge and $\mathrm{C}(s) = |s|^{p}$ with $p > 1$ yields the conservation of the $L^{p}$ norm of $f$. Two others important physical quantities, provided they are well-defined, are conserved. The total current verifies
\begin{align} \label{conservation_mass}
\forall t \in [0,T] , \quad    \int_{\T \times \R} f(t,x,v) v \dd v \dd x = \int_{\T \times \R} f_{0}(x,v) v \dd v \dd x.
\end{align}
The total energy of the system \eqref{Vlasov--Poisson_eq} defined for every $t \in [0,T]$ by
\begin{align} \label{total_energy_vp}
    \mathcal{E}(t) := \int_{\T \times \R} \frac{v^2}{2}f(t,x,v) \dd v \dd x + \int_{\T} \frac{1}{2} \big | E(t,x) \big |^2 \dd x
\end{align}
is also conserved.
 
\section{Multi-fluid solutions}\label{sec:multi-fluid_solutions}
\subsection{Existence}
A well-known class of weak solutions to the Vlasov--Poisson system is given by the so-called \textit{multi-fluid} Ansatz
\begin{align} \label{eq:multifluid_F}
    f^{N}(t,x,v) = \sum_{\ell \in I_{N} } \rho^{\ell}(t,x) \o \delta_{v = u^{\ell}(t,x)}, \quad N \in \N,
\end{align}
where $I_N$ is a finite or infinite index set and $(\rho^\ell,u^\ell)$ denote the density and velocity of the $\ell$-th fluid with $\rho^{\ell} \geq 0$, where the fluids $(\rho^{\ell},u^{\ell})_{\ell \in I_{N}}$ satisfy a pressureless Euler--Poisson system (see \eqref{M_alpha} below).

The following result, which is a variant in dimension $d =1$ of Theorem 1.3.11 in \cite{BEH}, shows that the Ansatz \eqref{eq:multifluid_F} yields a unique weak solution of the Vlasov--Poisson system whenever the fluid variables solve a coupled pressureless Euler--Poisson system and the velocities remain Lipschitz in space. In other words, it establishes the local-in-time equivalence between the Vlasov--Poisson and pressureless Euler--Poisson systems for this class of initial data. More precisely, we have.

\begin{theoreme}[Multi-fluid solution] \label{main_theorem}  Fix $N \in \N$ and let $f_0^{N}$ be a non negative initial distribution function in the form \eqref{eq:multifluid_F} with initial data
$(\rho^{\ell}_{0},u^{\ell}_{0})_{\ell \in I_{N}} \subset W^{m,\infty}(\T) \times W^{m+1,\infty}(\T)$ for some $m \geq 0$ and such that 
 $$ \displaystyle \sum_{\ell \in I_{N}} \| \rho^{\ell}_{0}\|_{W^{m,\infty}(\T)}<+\infty, \qquad \sup_{\ell \in I_{N}}  \:  \| u^{\ell}_{0}- \lambda^\ell\|_{W^{m+1,\infty}(\T)}<+\infty, $$
for some sequence $\lambda: I_{N} \to \R$. 
Then, for every $0 \leq \delta < L$, there exists $T_{N,\delta} > 0$ and a unique weak solution $f^{N}_{\delta}$ of the Vlasov--Poisson system $\textnormal{(VP)}_{\delta}$ on $[0,T_{N,\delta}]$. Moreover, $f^{N}_{\delta}$ takes the form \eqref{eq:multifluid_F} where 
$(\rho_{\delta},u_{\delta})=(\rho^{\ell}_{\delta},u^{\ell}_{\delta})_{\ell \in I_{N}}\subset \mathscr{C}\Big( [0,T_{N,\delta}];W^{m,\infty}(\T)) \times W^{m+1,\infty}(\T)) \Big) \cap \mathscr{C}^{1}\Big( [0,T_{N,\delta}];W^{m-1,\infty}(\T)) \times W^{m,\infty}(\T)) \Big)$ is the unique solution of the pressureless Euler--Poisson system
\begin{equation} \label{M_alpha}
(M_{\ell,\delta}):
\left|
\begin{array}{cc}
&  \partial_{t} \rho^{\ell}_{\delta} + \partial_{x}(\rho^{\ell}_{\delta} u^{\ell}_{\delta}) = 0,\\
&  \partial_{t} (\rho^{\ell}_{\delta} u^{\ell}_{\delta}) + \partial_{x} (\rho^{\ell}_{\delta} (u^{\ell}_{\delta})^2) = \rho^{\ell}_{\delta} E_{\delta},\\
\end{array}
\right.
\end{equation}
\begin{equation} 
     E_{\delta}(t,\cdot) = K_{\delta} \star \Big( \sum_{\ell \in I_{N}} \rho^{\ell}_{\delta}(t,\cdot) -1 \Big),\label{Poisson}
\end{equation}
with initial condition 
\begin{equation}\label{initial_phases}
    \rho^{\ell}_{\delta}(0,x) = \rho^{\ell}_{0}(x), \quad u^{\ell}_{\delta}(0,x) = u^{\ell}_{0}(x), \quad x\in \T,
\end{equation}
and satisfies
\begin{equation}
 \sup_{t \in [0,T_{N,\delta}]} \displaystyle \sum_{\ell \in I_{N}} \| \rho^{\ell}_{\delta}(t)\|_{W^{m,\infty}(\T)}<+\infty, \qquad  \sup_{t \in [0,T_{N,\delta}]} \sup_{\ell \in I_{N}}   \| u^{\ell}_{\delta}(t)-\lambda^\ell\|_{W^{m+1,\infty}(\T)}<+\infty. \label{uniform_estimate_ell}
\end{equation}
Moreover, $T_{N,\delta}$ is such that 
\begin{equation}\label{def_existence_time}
 T_{N,\delta}  \geq  \sup \Big\lbrace \tau > 0 \: : \:  \underset{\ell \in I_{N}} \sup \int_{0}^{\tau} \| \partial_{x} u^{\ell}_{\delta}(t) \|_{L^{\infty}(\T)} \dd t < +\infty \Big \rbrace. 
\end{equation}
\end{theoreme}
The role of the sequence $\lambda$ in this theorem is to allow velocities $(u_0^\ell)_{\ell \in I_N}$ that are decaying perturbations of unbounded constant reference velocities.
The proof relies on the correspondence between the Lagrangian formulations of the Vlasov--Poisson and pressureless Euler--Poisson systems. Since this Lagrangian viewpoint also underlies the numerical method developed in a forthcoming paper, we briefly recall the main arguments.
We first show that the multi-fluid representation \eqref{eq:multifluid_F} solves the Vlasov equation associated with the electric field generated by the Euler--Poisson system. We then verify that the induced charge density satisfies the Poisson equation, thereby recovering the Vlasov--Poisson system. Uniqueness of weak solutions \cite{LOEPER200668} finally implies that the solution remains of the \textit{multi-fluid} form \eqref{eq:multifluid_F}.
\begin{proof}
 This result, in particular the existence and uniqueness statement for the Euler--Poisson system \eqref{M_alpha}--\eqref{Poisson}, follows from \cite[Corollary 3.2.5]{BEH}, with a minor difference: as we work in 1D, one can replace the $L^2$ Sobolev framework of \cite{BEH} by the $L^\infty$ Sobolev spaces considered here. For pedagogical purposes, let us explain how a solution to the Euler--Poisson system indeed gives rise to a solution to Vlasov--Poisson.
Fix $0 \leq \delta < L$ and denote $(\rho^{\ell}_{\delta},u^{\ell}_{\delta})_{\ell \in I_{N}}$ be the solution the Euler--Poisson system \eqref{M_alpha}--\eqref{Poisson}. For each $\ell \in I_{N}$, let  define the Lagrangian flow by
\begin{equation}
\label{dpsidt}
   \frac{d}{dt} \psi_{t,\delta}^{\ell}(x) = u^{\ell}_{\delta}(t,\psi_{t,\delta}^{\ell}(x)) , \quad \psi_{0,\delta}^{\ell}(x) = x.
\end{equation}
Since $u^{\ell}_{\delta}$ is continuous in its first variable and Lipschitz continuous in its second variable, the map $\psi_{t,\delta}^{\ell}$ is well defined on $[0,T_{N,\delta}] \times \T$ and the map $x \longmapsto \psi_{t,\delta}(x)$ is a Lipschitz homeomorphism. Differentiating \eqref{dpsidt} once more in time and using the momentum equation \eqref{M_alpha} gives
$$
\frac{d^2}{dt^2} \psi_{t,\delta}^{\ell}(x) = E_{\delta}(t,\psi_{t,\delta}^{\ell}(x)), \quad \frac{d}{dt} \psi_{0,\delta}^{\ell}(x) = u_{\delta}^{\ell}(0,x).
$$ 
Hence each fluid trajectory satisfies the characteristic equations of the Vlasov equation \eqref{eq:charac} with initial data $(x,u^\ell_{\delta}(0,x))$. By uniqueness of solutions to the characteristic system we infer,
\begin{equation}
    \psi_{t,\delta}^{\ell}(x) = \hat{X}_{\delta}(t,0,x,u^{\ell}_{\delta}(0,x)), \quad \frac{d}{dt} \psi_{t,\delta}^{\ell}(x) = \hat{V}_{\delta}(t,0,x,u^{\ell}_{\delta}(0,x)),
\end{equation}
where 
\begin{equation}
\label{eq:charac}
\left|
\begin{array}{cll}
   & \dot{\hat{X}}_{\delta}(t,0,x,v) = \hat{V}_{\delta}(t,0,x,v), \:  t\in [0,T_{N,\delta}] \\
    & \dot{\hat{V}}_{\delta}(t,0,x,v) = E_{\delta}(t,\hat{X}_{\delta}(t,0,x,v)), \: t \in [0,T_{N,\delta}]\\
   & \hat{X}_{\delta}(0,0,x,v) = x, \quad \hat{V}_{\delta}(0,0,x,v) = v.
\end{array}
\right.
\end{equation}
Denote by
\begin{equation*}
    \hat{\varphi}_{t,0}^{\delta} : (x,v) \longmapsto (\hat{X}_{\delta}(t,0,x,v), \hat{V}_{\delta}(t,0,x,v))
\end{equation*}
the corresponding characteristic flow.
To prove that $f^{N}_{\delta}(t)$ given in \eqref{eq:multifluid_F} solves the Vlasov--Poisson system, recall that the continuity equation implies that 
\begin{equation} \label{def_rho_t}
    \rho^{\ell}_{\delta}(t) = \psi_{t,\delta}^{\ell} \# \rho^{\ell}_{0}
\end{equation}
and that the initial data is independent of $\delta.$
Therefore, for every $\xi \in \mathscr{C}_{b}(\T \times \R),$
\begin{align*}
    \int_{\T \times \R} \xi(x,v)  f^{N}_{\delta}(t,\dd x,\dd v) &= \sum_{\ell \in I_{N}} \int_{\T} \xi(x,u^{\ell}_{\delta}(t,x)) \rho^{\ell}_{\delta}(t,x) \dd x\\
    &=  \sum_{\ell \in I_{N}} \int_{\T} \xi(x,u^{\ell}_{\delta}(t,x)) \rho^{\ell}_{\delta}(t)(\dd x) \\
    & = \sum_{\ell \in I_{N}} \int_{\T} \xi(\psi_{t,\delta}^{\ell}(x),u^{\ell}_{\delta}(t,\psi_{t,\delta}^{\ell}(x)) \rho_{0}^{\ell}(\dd x)\\
    & = \sum_{\ell \in I_{N}} \int_{\T} \xi\big(\hat{X}_{\delta}(t,0,x,u^{\ell}_{\delta}(0,x)),
    \hat{V}_{\delta}(t,0,x,
    u^{\ell}_{\delta}(0,x)\big)
    \rho_{0}^{\ell}(\dd x) \\
    &=\sum_{\ell \in I_{N}} \int_{\T} \xi\circ  \hat{\varphi}^{\delta}_{t,0}(x, u^\ell_0(x))  \rho_{0}^{\ell}(\dd x)\\
    & =\int_{\T \times \R} \xi(x,v) \hat{\varphi}^{\delta}_{t,0} \# f^{N}_{0}(\dd x,\dd v).
\end{align*}
Hence, 
\begin{align*}
 f^{N}_{\delta}(t) = \hat{\varphi}^{\delta}_{t,0} \# f^{N}_{\delta}(0)   
\end{align*}
so $f^{N}_{\delta}(t)$ satisfies the Vlasov equation
$$
\partial_{t} f^{N}_{\delta} + v \partial_{x} f^{N}_{\delta} + E_{\delta} \partial_{v} f^{N}_{\delta} = 0.
$$
Finally,
$$
\int_{\R} f^{N}_{\delta}(t,x,\dd v ) =  \sum_{\ell \in I_{N}} \rho^{\ell}_{\delta}(t,x),
$$
and the Poisson equation \eqref{Poisson} becomes
$$
 E_{\delta}(t,\cdot)= K_{\delta} \star \Big( \displaystyle \int_{\R} f^{N}_{\delta}(t,\cdot,\dd v)  -1\Big).
$$
Therefore, $(f^N_{\delta},E_{\delta})$ solves the Vlasov--Poisson system. By uniqueness of weak solutions, the characteristic flow $\hat{\varphi}^{\delta}_{t,0}$ coincides with the nonlinear Vlasov--Poisson flow $\varphi^{\delta}_{t,0}$ defined in \eqref{eq:nonlinear_flow_VP}.
\end{proof}
An immediate consequence of the fact that the \textit{multi-fluid} solution $f^{N}_{\delta}$ solves a Vlasov--Poisson type system is the conservation of mass, momentum and energy, provided they are well-defined. Since the mass and momentum are independent of the regularization parameter $\delta$ at the initial time, they remain independent of $\delta$ for all subsequent times.

\begin{proposition}[Conservations] Let $0 \leq \delta < L$. Fix $N \in \N$ and let $f^{N}_{\delta}$ be a \textit{multi-fluid} solution on $[0,T_{N,\delta}]$ with $f^{N}_{\delta}(0)$ independent of $\delta.$ Then, it holds for $t \in [0,T_{N,\delta}],$
\begin{align}
    &\int_{\T} \rho^{\ell}_{\delta}(t,x) \dd x = \int_{\T} \rho^{\ell}(0,x) \dd x, \quad \ell \in I_{N} \label{conv_mass},\\
    &\int_{\T} \sum_{\ell \in I_{N}} \rho^{\ell}_{\delta}(t,x)u^{\ell}_{\delta}(t,x) \dd x= \int_{\T} \sum_{\ell \in I_{N}} \rho^{\ell}(0,x)u^{\ell}(0,x) \dd x, \label{conv_momentum}\\
    &\int_{\T} \sum_{\ell \in I_{N}}  \frac{\rho^{\ell}_{\delta}(t,x) u^{\ell}_{\delta}(t,x)^2}{2} + \frac{E_{\delta}(t,x)^2}{2} \dd x = \int_{\T} \sum_{\ell \in I_{N}}  \frac{\rho^{\ell}(0,x) u^{\ell}(0,x)^2}{2} + \frac{E_{\delta}^2(0,x)}{2} \dd x. \label{conv_energy}
\end{align}
\end{proposition}
\subsection{Growth rate of the Lipschitz norm of the fluid velocities} \label{sec:evol_lipschitz_norm}
A central feature of the \textit{multi-fluid} method we propose is that, at the initial time, all fluid velocities are chosen to be constant, and therefore have initially vanishing Lipschitz norms. Consequently, these norms cannot grow instantaneously, provided that the Lipschitz norm of the electric field remains bounded. For our analysis, however, we need a stronger property: we require the Lipschitz norms of the fluid velocities to grow at most quadratically, embodying the fact that the dominant part of the dynamics is the fluid one.  This, in turn, requires a uniform-in-time bound on the Lipschitz norm of the electric field. When $\delta=0$, such a bound cannot be obtained for \textit{multi-fluid} solutions since they are irregular in velocity. When $\delta>0$ however, this becomes possible, since the regularized Coulomb kernel appearing in the Poisson equation is Lipschitz continuous, while the total mass is readily preserved by the dynamics. The resulting existence time is of order $\sqrt{\delta}$. We establish the following result.
\begin{proposition}[Growth of the Lipschitz norm of the fluid velocities] \label{lower_bound_local_existence_time}
Let $0 < \delta < L$. Fix $N \in \N$ and let $f^{N}_{\delta}$ be a \textit{multi-fluid} solution on $[0,T_{N,\delta}]$ given by Theorem \ref{main_theorem} with an initial data $(\rho_{0}^{\ell},u_{0}^{\ell})$ such that
\begin{align}
      \partial_{x} u^{\ell}_{0}  = 0, \: \ell \in I_{N}.
\end{align}
Then, for every $0 < \eta < 1$ 
\begin{align}
    T^{N}_{\delta, \eta}:= \sup \Big \lbrace t \geq 0 \: :  \underset{ \ell \in I_{N}} \sup \int_{0}^{t} \| \partial_{x} u^{\ell}_{\delta}(\tau) \|_{L^{\infty}(\T)} \dd \tau \leq \eta  \Big \rbrace  \label{lower_bound_s1}
\end{align}
satisfies the lower bound
\begin{align}
    T^{N}_{\delta,\eta} \geq \sqrt{\frac{ \eta(1-\eta) \delta}{ M_{0} + L}},  \label{lower_bound_time_multi_fluid}
\end{align}
where $M_{0} = \displaystyle \sum_{\ell \in I_{N}} \int_{\T} \rho^{\ell}_{0}(x) \dd x.$

\begin{proof} Differentiating in space the momentum equation yields for every $\ell \in I_{N}$ and $0 \leq t \leq T^{N}_{\delta,\eta},$

\begin{align*}
\partial_{t} (\partial_{x} u^{\ell}_{\delta}) + u^{\ell}_{\delta} \partial^2_{x} u^{\ell}_{\delta} = \partial_{x} E_{\delta} - |\partial_{x} u^{\ell}_{\delta}|^2.
\end{align*}
Now, we use the Lagrangian form of the equation
\begin{align*}
    \partial_{x} u^{\ell}_{\delta}(t,x) &= \partial_{x} u^{\ell}_{0}( (\psi_{t,\delta}^{\ell})^{-1}(x)) + \int_{0}^{t} \partial_{x} E_{\delta}(\tau, \psi_{\tau,\delta}^{\ell} \circ (\psi_{t,\delta}^{\ell})^{-1}(x) )\dd \tau\\
    &-\int_{0}^{t} | \partial_{x} u^{\ell}_{\delta}  |^2\big( \tau,\psi_{\tau,\delta}^{\ell} \circ (\psi_{t,\delta}^{\ell})^{-1}(x) \big) \dd \tau.
\end{align*}
Since initially the fluid velocities are constant,  setting $y_{\delta}(t) = \underset{ \ell \in I_{N}} \sup \| \partial_{x} u^{\ell}_{\delta} (t) \|_{L^{\infty}(\T)}$  we get
\begin{align}
\label{estime1_y}
    y_{\delta}(t) \leq \int_{0}^{t} \| \partial_{x} E_{\delta}(\tau) \|_{L^{\infty}(\T)} \dd \tau+ \int_{0}^{t} y_{\delta}(\tau)^2 \dd \tau.
\end{align}
To obtain a closed estimate, we need to estimate the electric field. Using the estimate \eqref{bound_Lipschitz_K} we have the bound
\begin{align}
    \| \partial_{x} E_{\delta}(\tau) \|_{L^{\infty}(\T)} \leq \frac{M_0 + L}{\delta},
\end{align}
and inserting it in \eqref{estime1_y} yields
\begin{align}
    y_{\delta}(t) \leq \int_{0}^{t} y_{\delta}(\tau)^2 \dd \tau +  t\Big(\frac{M_{0} + L}{\delta}\Big). \label{gronwall_y}
\end{align}
For $r \geq 0$, we set $T_{r,\delta} := \sup \big \lbrace t \geq 0 \: : y_{\delta}(s) \leq r  \quad \forall s \in [0,t] \big \rbrace$. We thus get for $t \in [0,T_{r,\delta}]$
\begin{align}
    y_{\delta}(t) \leq  T_{r,\delta}\Big( r^2  + \frac{M_0 + L}{\delta}\Big). \label{estimate_2_y}
\end{align}
By continuity of $y_{\delta}$ (which is consequence of the uniform estimate in $\ell$ of Theorem \ref{main_theorem}) and the definition of $T_{r,\delta}$ we obtain
\begin{align} \label{ineq_TM}
    r \leq T_{r,\delta}\Big(r^2 + \frac{M_0 + L}{\delta}\Big) \Longrightarrow T_{r,\delta} \geq \frac{r}{r^2 + \frac{M_0+L}{\delta}}=: t^{\star}(r,\delta).
\end{align}
Taking
$$
T_{r_{\delta},\delta} =t^{\star}(r_\delta,\delta)
$$
with
$$
     r_{\delta}  := \sqrt{ \frac{\eta (M_0 + L)}{ (1-\eta) \delta }  }, 
$$
we obtain
$$
T_{r_{\delta},\delta} = \sqrt{\frac{ \eta(1-\eta) \delta}{ M_{0} + L}}
$$
and
$$
T_{r_{\delta},\delta} r_\delta =\eta.
$$
Eventually, since $\underset{ \ell \in I_{N} } \sup \int_{0}^{T_{{r}_{\delta},\delta}} \| \partial_{x} u^{\ell}_{\delta}(s) \|_{L^{\infty}(\T)} \dd s \leq \int_{0}^{T_{r_{\delta},\delta }} y_{\delta}(s) \dd s \leq \eta$, we deduce that $T^{N}_{\delta,\eta} \geq T_{{r_\delta},\delta}$, which yields \eqref{lower_bound_time_multi_fluid}.

\end{proof}
\end{proposition}

\section{Wasserstein-1 stability estimate}
\label{sec:wasserstein-stab}
In this section, we discuss a stability estimate in the Wasserstein-1 distance between weak solutions of the Vlasov--Poisson system. To this end, we recall (see \cite{Santambrogio}) that the Wasserstein-1 distance between two probability measures $\mu,\nu\in\mathcal P_1(X)$, where $X=\mathbb T\times\mathbb R$ is endowed with its natural distance $d$, is defined by
\begin{equation}
\mathcal W_1(\mu,\nu)
:=
\inf_{\gamma\in\Pi(\mu,\nu)}
\int_{X\times X} d (z,\hat{z}) \dd \gamma(z,\hat{z}),
\end{equation}
where $\Pi(\mu,\nu)$ denotes the set of transport plans between $\mu$ and $\nu$, namely the probability measures on $X\times X$ whose first and second marginals are $\mu$ and $\nu$, respectively. Here, $\mathcal P_1(X)$ denotes the space of probability measures on $X$ with finite first moment. A distinguished subclass of transport plans consists of those induced by a measurable transport map $T:X\to X$, namely
\begin{equation}
\mathcal{C}
:=
\Big \lbrace
\gamma\in\Pi(\mu,\nu):
\exists \: T:X\to X \: :
\gamma(\dd z,\dd \hat{z}) = \mu(\dd z) \otimes \delta_{T(z)}(\dd \hat{z})
\Big \rbrace.
\end{equation}
We thus obtain an easy bound of the Wasserstein-1 distance by the so-called Monge cost (see e.g. \cite{Santambrogio})
\begin{equation} \label{eq:bound_W1_by_Monge}
    \mathcal{W}_{1}(\mu,\nu) \leq\underset{  T : X \longrightarrow X,\\ \nu = T \# \mu}{ \inf } \int_{X} d(z,Tz)  \mu(\dd z).
\end{equation}
The following statement, which follows from a result of Hauray \cite{Hauray}, provides a weak-strong stability estimate in the Wasserstein-1 distance for the Vlasov--Poisson system, a property specific to dimension $1$ (see \cite{LOEPER200668} for extensions to higher dimensions). More precisely, it shows that, in order to compare two weak solutions of the Vlasov--Poisson system, it is sufficient for one of the solutions to have a  bounded charge density at each time, while the other only needs to have a finite first moment.
\begin{proposition} \label{Maxime_Hauray_theorem}
Fix $T > 0$ and $f_0 \in \mathcal{A}$. Then, for every $0 \leq \delta < L$, the unique weak solution $f_{\delta}$  on $[0,T]$ to $\textnormal{(VP)}_{\delta}$ with initial data $f_0$ satisfies the boundedness of the charge uniformly with respect to $\delta,$
\begin{align}
  & \underset{\delta \geq  0} \sup \| \rho_{f_{\delta}}(t) \|_{L^{\infty}(\T)} \leq \max \lbrace g_0(0), \| g_0\|_{L^1} \rbrace \Big(1+ t( \|f_0\|_{L^{1}} + L)\Big), \quad \textnormal{ for } t \in [0,T],\label{bound_rho_f}
\end{align}
and the pointwise estimate
\begin{align}
  & f_{\delta}(T,x,v) \leq g_{T}(|v|)\quad \textnormal{ for a.e } (x,v) \in \T \times \R, \label{decay_f_in_v}
\end{align}
where $g_{T}$ is a non increasing function in $L^{1}(\R^{+})$ which depends only on the initial data. In addition, for any other measure solution $\nu_{\delta}$ of $\textnormal{(VP)}_{\delta}$ on $[0,T]$ such that
\begin{align*}
    & \int_{\T \times \R} |v| \nu_{\delta}(0,\dd x, \dd v) < +\infty,
\end{align*}
the following stability estimate holds
\begin{equation}
    \label{eq:wasserstein-stab}
    \mathcal{W}_1(f_{\delta}(t), \nu_{\delta}(t)) \leq e^{a(t)} \mathcal{W}_1(f_{\delta}(0), \nu_{\delta}(0)),
\end{equation}
where
\begin{align}
\quad a(t):= 8 \int_{0}^{t} (1 +  \underset{\delta \geq 0} \sup \:  \| \rho_{f_{\delta}}(s) \|_{L^{\infty}(\T)}) \dd s. \label{function_a_t}
\end{align}
\begin{proof} We only give the proof of \eqref{bound_rho_f}, \eqref{decay_f_in_v} as the stability estimate \eqref{eq:wasserstein-stab} is precisely the one given in Theorem 1.9 of \cite{Hauray}.
We recall that for $t \in [0,T]$ we have
\begin{align*}
    f(t,x,v) = f_0 ( X_{\delta}(0,t,x,v), V_{\delta}(0,t,x,v) ) \quad \textnormal{ for a.e } (x,v) \in \T \times \R,
\end{align*}
where $(X_{\delta},V_{\delta})$ are the characteristics of the Vlasov--Poisson system $\textnormal{(VP)}_{\delta}.$ By Definition \ref{admissible_initial_data} of $ \mathcal{A}$, there exists a non-increasing function $g_0 : \R^{+} \longrightarrow \R^{+}$ in $\to  L^1(\R^+;(1+|v|)dv)$ such that $f_0(x,v) \leq g_0(|v|).$ So we get,
\begin{align*}
    f_{\delta}(t,x,v) \leq g_0 (|V_{\delta}(0,t,x,v)|).
\end{align*}
Using the characteristics we have $v- V_{\delta}(0,t,x,v)  = \displaystyle \int_{0}^{t} E_{\delta}(s,X_{\delta}(s,t,x,v)) \dd s.$
Thus,
\begin{align*}
    |v - V_{\delta}(0,t,x,v)| \leq \int_{0}^{t} \| E_{\delta}(s) \|_{L^{\infty}(\T)} \leq \frac{t}{2} (\| f_0 \|_{L^{1}(\T \times \R)} + L),
\end{align*}
where we used the estimate \eqref{bound_Lipschitz_K} to get the last bound. By triangular inequality we have
\begin{align*}
    |V_{\delta}(0,t,x,v)| \geq |v| - \frac{t}{2} (\| f_0 \|_{L^{1}(\T \times \R)} + L).
\end{align*}
For $|v| \leq \frac{t}{2} (\| f_0 \|_{L^{1}(\T \times \R)} + L)$ we have $g_0(|V_{\delta}(0,t,x,v)|) \leq g_0(0)$ and for $|v| > \frac{t}{2} (\| f_0 \|_{L^{1}(\T \times \R)} + L)$ we have $g_0(|V_{\delta}(0,t,x,v)|) \leq g_0\left( |v| -\frac{t}{2} (\| f_0 \|_{L^{1}(\T \times \R)} + L) \right).$
We therefore get
\begin{equation}
\begin{aligned}
    \rho_{f_{\delta}}(t, x) = \int_{\R} f_{\delta}(t,x,v) \dd v  &\leq \int_{\R} g_0(|V_{\delta}(0,t,x,v)|) \dd v \\
    &\leq t g_0(0) (\| f_0 \|_{L^{1}(\T \times \R)} + L) + \|g_0\|_{L^{1}(\R^{+})}.
\end{aligned}
\end{equation}
It yields \eqref{bound_rho_f}.  We eventually glean \eqref{decay_f_in_v} with
\begin{align*}
g_{T}(|v|) = g_0(0) \mathrm{1}_{|v| < \frac{T}{2} (\| f_0 \|_{L^{1}(\T \times \R)} + L)} + \mathrm{1}_{|v| \geq \frac{T}{2} (\| f_0 \|_{L^{1}(\T \times \R)} + L)} g_0\left( |v| -\frac{T}{2} (\| f_0 \|_{L^{1}(\T \times \R)} + L)\right).\nonumber
\end{align*}
\end{proof}

\end{proposition}

\section{The \textit{multi-fluid} approximation}
\label{sec:def_multi-fluid_approx_main_result}
\subsection{Approximation of the initial data}
The first step of our method consists in approximating any $f_0 \in \mathcal{A}$ (recall Definition~\ref{admissible_initial_data}) by a \textit{multi-fluid} distribution of the form \eqref{eq:multifluid_F}, with a quantitative control of the approximation error. A natural strategy is to introduce a sequence of uniform velocity grids $(v^N)^{N \in \mathbb{N}^{\star}}$ with grid size $h:= \frac{(B-A)}{N}$ defined for every $N \in \N^{\star},$ by
\begin{align}
    v_{\ell+\frac{1}{2}}^{N} = v_{\ell}^{N} + \frac{h}{2} \quad \textnormal{ where } v_{\ell}^{N} = A + \ell h \textnormal{ for any } \ell \in \Z,
\end{align}
where $A< B$ are fixed numerical parameters. A natural approximation, seemingly related to the Multi-Stream approach \cite{Ghizoo}, consists in defining
\begin{align}
\label{f0:sec5}
    &f_0^{N}(x,v) = \sum_{\ell \in \Z} \rho_0^{\ell}(x) \o \delta_{v = u_0^{\ell}(x)},
\end{align}
where
\begin{align}
    &\rho_{0}^{\ell}(x) = \int_{v_{\ell-\frac{1}{2}}^{N}}^{v_{\ell+\frac{1}{2}}^{N}} f_0(x,v) \dd v,\quad  \rho_{0}^{\ell}(x) u_{0}^{\ell}(x) = \int_{v_{\ell-\frac{1}{2}}^{N}}^{v_{\ell+\frac{1}{2}}^{N}} v f_0(x,v) \dd v. \label{loc_J}
\end{align}
This construction is known to break down as soon as the Euler--Poisson dynamics makes at least one fluid velocity becomes a multi-valued function of $x$. 
Consequently, such an approximation is unsuitable for long-time simulations. Our approach differs from the construction \eqref{f0:sec5}-\eqref{loc_J}. We approximate $f_0$ by its push-forward under a fixed transport map associated with the velocity grid. More precisely, the transport map we consider sends all the mass located in the cell
$
\lbrace x \rbrace \times
[v_{\ell-\frac{1}{2}}^{N},v_{\ell+\frac{1}{2}}^{N})
$
to its center, namely the point
$
( x ,v_\ell^{N} ).
$
This may be seen as an approximation of the local current densities \eqref{loc_J} where all particles are assumed to have the same constant velocity in the cell.
A proper definition of the transport map is
\begin{equation} \label{def_transport_map}
    G^{N}(x,v) = (x,v_{\ell}^{N}) \text{ if } v_{\ell-\frac{1}{2}}^{N} \leq v < v_{\ell+\frac{1}{2}}^{N}, \quad \textnormal{ for any } (x,v) \in \T \times \R.
\end{equation}
When the initial data is a non negative function $f_{0} \in L^{1}(\T \times \R)$, its push-forward by the map $G^{N}$ is given by 
\begin{align} \label{eq:multi_fluid_initial}
    &G^{N} \# f_0 (x,v) = \sum_{\ell \in \Z} \rho^{\ell}_{0}(x) \o \delta_{v = v_{\ell}^{N}}, \quad \textnormal{ for any } N \in \N^{\star},
\end{align}
where
\begin{align} \label{eq:multi_fluid_initial_2}
    \rho^{\ell}_{0}(x) = \int_{v_{\ell-\frac{1}{2}}^{N}}^{v_{\ell+\frac{1}{2}}^{N}} f_{0}(x,v) \dd v, \quad \ell \in \Z.
\end{align}
This is exactly a \textit{multi-fluid} representation \eqref{eq:multifluid_F} where the initial fluid velocities are constant. One can note that $G^{N}$ is a linear projection operator in the sense that
\begin{align}
    \forall f \in L^{1}(\T \times \R), \: f \geq 0, \quad G^{N} \# G^{N} \# f  = G^{N} \# f.
\end{align}
It is now straightforward to prove that such an approximation actually converges to $f_0$ when the grid is refined. More precisely we have:
\begin{lemme}[Projection error] \label{projection_error} For every non negative $f \in L^{1} (\T \times \R; (1+|v|)\dd x \dd v)$ and $N \in \N^{\star},$ we have
\begin{align} \label{initial_error_estimate}
   & \mathcal{W}_{1}(G^{N} \# f, f) \leq \frac{h}{2} \|f\|_{L^{1}(\T \times \R)}.
\end{align}
\begin{proof} Thanks to the bound \eqref{eq:bound_W1_by_Monge}, there holds
\begin{align*}
    \mathcal{W}_{1}(G^{N} \# f,f) \leq \int_{\T \times \R} d  \big( (x,v), G^{N}(x,v) \big) f(x,v) \dd x \dd v. 
\end{align*}
A direct computation yields,
\begin{align*}
   & \int_{\T \times \R} d  \big( (x,v), G^{N}(x,v) \big) f(x,v) \dd x \dd v \\
   &\;\;\;\;\;\;\;\;\;\;\;\;= \sum_{\ell \in \Z} \int_{\T \times [v_{\ell-\frac{1}{2}}^{N},v_{\ell+\frac{1}{2}}^{N}) } |v-v_{\ell}^{N}| f(x,v) \dd x \dd v \leq \frac{h}{2} \| f \|_{L^{1}(\T \times \R)},
\end{align*}
hence the lemma.
\end{proof}
\end{lemme}
A consequence of Propositions \ref{lower_bound_local_existence_time} and \ref{Maxime_Hauray_theorem} is the propagation of the initial error on a short time interval.
\begin{proposition}[Short time convergence of the \textit{multi-fluid} approximation] Let $f_0 \in \mathcal{A}$ (recall Definition~\ref{admissible_initial_data}). Then, for every $0 \leq \delta < L$, $0 \leq \eta \leq 1$ and $N \in \N^{\star},$ we have,
\begin{align} \label{eq:multi_fluid_error}
    &\forall t \in [0,T^{N}_{\delta,\eta}], \quad \mathcal{W}_{1}(f_{\delta}(t),f^{N}_{\delta}(t)) \leq e^{a(t)} \mathcal{W}_{1}(f_0,f^{N}_{\delta}(0)),
\end{align}
where
\begin{align}
    &T_{\delta,\eta}^{N} := \sup \big\lbrace t \geq 0 \: : \:  \sup_{\ell \in \Z} \int_{0}^{t} \| \partial_{x} u^{\ell}_{N,\delta}(s) \|_{L^{\infty}(\T)} \dd s \leq \eta \big \rbrace,
\end{align}
and $f_{\delta}$ and $f^{N}_{\delta}$ denote respectively the solutions to the Vlasov--Poisson $\textnormal{(VP)}_{\delta}$ with respective initial data $f_0$ and $f^{N}_{\delta}(0) = G^{N} \# f_0$ and where $a(t)$ is given by \eqref{eq:wasserstein-stab}.
\end{proposition}
Thus, as long as the fluid velocities remain Lipschitz in space, we propagate the initial approximation error (of order $h$) up to an exponential factor in which the rate of increase is independent of $\delta > 0$ thanks to Proposition \ref{Maxime_Hauray_theorem}. 

\subsection{Remapping}\label{sec:remapping}
A drawback of the \textit{multi-fluid} formulation is the appearance of shocks in the Burgers dynamics for the velocities. These shocks are bound to occur when the force is non zero (even if it decays in time), but they shouldn't be seen as a true singularity: they only correspond to a failure of the \textit{multi-fluid} parametrization.
To overcome this issue, we define a remapping operator which consists in restarting the \textit{multi-fluid} parametrization before a possible shock. Before a shock, the \textit{multi-fluid} solution belongs to the set
\begin{align}
    K := \Big \lbrace \sum_{\ell \in \Z} \rho^{\ell}(x) \otimes \delta_{v = u^{\ell}(x)}  \: : \: (\rho,u) \in \big( L^{\infty}(\T;\R^{+}) \times W^{1,\infty}(\T)\big)^{\Z}  \Big \rbrace.
\end{align}
\begin{definition}[Remapping operator]
\label{def:remap} Fix $N \in \N^{\star}$.
The remapping operator is defined for every $g\in K$ by
\begin{align}
    &\mathcal{R}^{N}g(x,v) = \sum_{k \in \Z} \rho^{k}(x) \o \delta_{v = v_{\ell(k,x)}^{N}},
\end{align}
where
\begin{align}
\ell(k,x) = \Big \lfloor \frac{u^{k}(x) - A}{h} + \frac{1}{2} \Big \rfloor \textnormal{ when } g(x, v)= \sum_{k \in \Z} \rho^{k}(x) \o \delta_{v = u^{k}(x)}.
\end{align}
\end{definition}
The remapping procedure consists in applying the operator $\mathcal{R}^{N}$ on the \textit{multi-fluid} approximation \eqref{eq:multifluid_F} at increasing discrete times $0 =: s_{0} < s_1 < \dots < s_r <  T \leq s_{r+1}$  each time before the appearance of shocks; a practical condition to go from the $k$-th remapping (starting at time $s_k$) to the $(k+1)$-th one (starting at time $s_{k+1}$) consists in defining $s_{k+1}\geq s_k$ as
\begin{align} \label{def_s_k}
s_{k+1}= \sup \left\{ t \geq s_k, \, \underset{ \ell \in \Z } \sup \int_{s_k }^{t} \| \partial_x u^{\ell}_{N,\delta}(t) \|_{L^{\infty}(\T)} \dd t \leq  \eta \right\}
\end{align}
where $0 < \eta < 1$ is a fixed numerical parameter. To control the error over an arbitrary time interval $[0,T]$, we shall propagate the initial approximation error and the remapping errors on each interval $[s_{k},s_{k+1}).$  We emphasize that the remapping procedure is not merely a numerical artifact, but is essential to reach long-time asymptotics. Its drawback is the additional approximation error introduced each time the remapping operator is applied. As a consequence, the approximate solution develops discontinuities at the remapping times $s_k$, $k\geq 1$. Each remapping thus generates an additional error, which we are able to quantify and control. 
\begin{proposition}[Stability properties of the remapping operator] For every $N \in \N^{\star}$ and $g(x,v) = \displaystyle \sum_{k \in \Z } \rho^{k}(x) \o \delta_{v = u^{k}(x)} \in K$ where $\rho^{k} \geq 0$ we have,
\begin{align}
&\mathcal{R}^{N} g (x,v) = G^{N} \#g (x,v) =  \sum_{\ell \in \Z} \tilde{\rho}^{\ell}(x) \o \delta_{v = v_{\ell}^{N}}, \label{remap_belongs_to_K}
\end{align}
where $ \tilde{\rho}^{\ell}(x)  = \sum_{k \in \Z} \rho^{k}(x) \mathrm{1}_{ v_{\ell-\frac{1}{2}}^{N} \leq u^{k}(x) < v_{\ell+\frac{1}{2}}^{N}}.$
Furthermore,
\begin{align}
&\int_{\R} \mathcal{R}^{N} g(x,\dd v)   =  \int_{\R} g(x,\dd v), \label{charge_preserving}\\
 & \int_{\T \times \R}  \mathcal{R}^{N} g(\dd x,\dd v)  = \int_{\T \times \R} g(\dd x ,\dd v), \label{mass_preserving}\\
 & \mathcal{W}_{1}(\mathcal{R}^{N} g,g)  \leq \frac{h}{2} g(\T \times \R) \label{w1_stab},\\
& \sum_{\ell \in \Z} \| \tilde{\rho}^{\ell} \|_{L^{\infty}(\T)} \leq \left( \frac{2 V}{h} + 1 \right) \sum_{\ell \in \Z} \| \rho^{\ell} \|_{L^{\infty}(\T)}, \quad V = \underset{ k \in \Z} \sup \| u^{k} - v_{k}^{N} \|_{L^{\infty}(\T)}. \label{estimate_sum_linf_after_remap}
\end{align}
\begin{proof}
Let $g \in K$ in the form $g(x,v) = \displaystyle \sum_{k \in \Z} \rho^{k}(x) \o \delta_{v = u^{k}(x)}$. Using Fubini-Tonelli we have for almost every $x \in \T$,
\begin{align*}
    \sum_{\ell \in \Z} \tilde{\rho}^{\ell}(x) \o \delta_{v = v_{\ell}^{N}} = \sum_{k \in \Z} \rho^{k}(x)  \sum_{\ell \in \Z} \mathrm{1}_{ v_{\ell-\frac{1}{2}}^{N} \leq u^{k}(x) < v_{\ell+\frac{1}{2}}^{N}} \o \delta_{v=v_{\ell}^{N}} = \mathcal{R}^{N}g(x,v).
\end{align*}
As a consequence a direct computation shows that $\mathcal{R}^{N}g = G^{N}\# g$. It proves \eqref{remap_belongs_to_K}. As for \eqref{charge_preserving}, we integrate in velocity and use again Fubini-Tonelli to get
\begin{align*}
&\int_{\R}  \mathcal{R}^{N} g(x,\dd v) = \sum_{\ell \in \Z} \sum_{k \in \Z} \rho^{k}(x) \mathrm{1}_{ v_{\ell-\frac{1}{2}}^{N} \leq u^{k}(x) < v_{\ell+\frac{1}{2}}^{N}} = \sum_{k \in \Z} \rho^{k}(x) \mathrm{1}_{-\infty \leq u^{k}(x) < +\infty} = \sum_{k \in \Z} \rho^{k}(x).
\end{align*}
We readily obtain \eqref{mass_preserving}.
To prove \eqref{w1_stab}, we use the fact that $\mathcal{R}^{N}g = G^{N} \# g$ so thanks to the bound \eqref{eq:bound_W1_by_Monge} we get
\begin{align*}
    \mathcal{W}_{1}(g,\mathcal{R}^{N}g) =  \mathcal{W}_{1}(g,G^{N}\#g) &\leq \int_{\T \times \R} d( (x,v); G^{N}(x,v)) g(\dd x, \dd v) \\
    &= \sum_{k \in \Z} \int_{\T} \rho^{k}(x) |u^{k}(x) - v_{\ell(k,x)}^{N}| \dd x,
\end{align*}
where $\ell(k,x) = \Big \lfloor \frac{u^{k}(x) - A}{h} + \frac{1}{2} \Big \rfloor. $ Since for every $k \in \Z$ and almost every $x \in \T,$ $|u^{k}(x) - v_{\ell(k,x)}^{N}| \leq \frac{h}{2}$ we obtain
\begin{align*}
    \mathcal{W}_{1}(\mathcal{R}^{N}g,g) \leq \frac{h}{2} g(\T \times \R).
\end{align*}
We prove the last property \eqref{estimate_sum_linf_after_remap}. Fix $\ell \in \Z$. Thanks to \eqref{remap_belongs_to_K}, for almost every $x \in \T$
\begin{align*}
    \tilde{\rho}^{\ell}(x) = \sum_{k \in \Z} \rho^{k}(x) \mathrm{1}_{v_{\ell-\frac{1}{2}}^{N} \leq u^{k}(x) < v_{\ell+\frac{1}{2}}^{N} }.
\end{align*}
The rest of the proof consists in counting the number of fluid velocies that are localized in the interval $[v_{\ell - \frac{1}{2}}^{N},v_{\ell + \frac{1}{2}}^{N}). $  By definition of $V$ we have for every $x \in \T,$
\begin{align*}
   -V +  v_{k}^{N} - v_{\ell}^{N} \leq  u^{k}(x) - v_{\ell}^{N} \leq v_{k}^{N} - v_{\ell}^{N} +  V.
\end{align*}
Therefore for every $ k \in \Z$,
\begin{align}
    (k-\ell)h - V \geq \frac{h}{2} \textnormal{ or } (k-\ell)h + V < -\frac{h}{2} \Longrightarrow \forall x \in \T, \mathrm{1}_{v_{\ell-\frac{1}{2}}^{N} \leq u^{k}(x) < v_{\ell+\frac{1}{2}}^{N} }  = 0.
\end{align}
By contraposition, we thus have for almost every $x \in \T,$
\begin{align*}
    \tilde{\rho}^{\ell}(x) \leq \sum_{k \in C(\ell)} \rho^{k}(x) \mathrm{1}_{v_{\ell-\frac{1}{2}}^{N} \leq u^{k}(x) < v_{\ell+\frac{1}{2}}^{N} }.
\end{align*}
where 
\begin{align}
    C(\ell) := \Big \lbrace k \in \Z \: : \:  (k-\ell)h - V < \frac{h}{2}   \textnormal{ and } (k-\ell)h + V \geq -\frac{h}{2} \Big \rbrace.
\end{align}
We thus get 
\begin{align}
    \| \tilde{\rho}^{\ell} \|_{L^{\infty}(\T)} \leq \sum_{k \in C(\ell)} \| \rho^{k} \|_{L^{\infty}(\T)},
\end{align}
By symmetry, observe that $k \in C(\ell)$ if and only if $ \ell \in C(k).$ 
Therefore, by an additional summation and using Fubini we get
\begin{align}
    \sum_{\ell \in \Z} \| \tilde{\rho}^{\ell} \|_{L^{\infty}(\T)}  \leq \sum_{k \in \Z}\| \rho^{k} \|_{L^{\infty}(\T)} \sum_{\ell \in C(k)} 1.
\end{align}
We have $\sum_{\ell \in C(k)} 1 \leq \frac{ 2V}{h} +1$ which eventually yields \eqref{estimate_sum_linf_after_remap}.
\end{proof}
\end{proposition}
Although the following regularization is not required for our analysis, it is customary, for numerical purposes, to smooth the Dirac masses by convolution with a kernel in the velocity variable, thereby producing numerical approximations of elements of the set $K$. This regularization could have been introduced into the remapping operator $\mathcal{R}^{N}$ at the price of an additional error, which can also be quantified and controlled.
\begin{lemme}[Regularization by convolution in the velocity variable] \label{convol_lemma}Let $\psi \in \mathscr{C}_{c}^{0}(\R)$ be non negative, even and of unit total mass. Let $( \psi_{\varepsilon}(\cdot) = \frac{1}{\varepsilon} \psi( \frac{\cdot}{\varepsilon}) )_{\varepsilon > 0}$ be the associated approximation of unity. For any $g \in K$ and $\varepsilon > 0$ we have
\begin{align} \label{regularization_estimate_v}
\mathcal{W}_{1}(g \star_{v} \psi_{\varepsilon}, g) \leq C_{\psi} \varepsilon g(\T \times \R),
\end{align}
where $C_{\psi} = \int_{\R} \psi(v) |v| \dd v$  and the convolution in the velocity variable is given 
\begin{align}
    (g \star_{v} \psi_{\varepsilon})(x,v) = \sum_{\ell \in \Z} \rho^{\ell}(x) \psi_{\varepsilon}(v-u^{\ell}(x)).
\end{align}
\begin{proof}
Observe that $g \star_{v} \psi_{\varepsilon}$ is a probability measure since $\psi_{\varepsilon}$ is of unit total mass. We use here the dual formulation of the Wasserstein-1 distance (see \cite{Santambrogio}),
\begin{align*}
    \mathcal{W}_1(g \star_{v} \psi_{\varepsilon}, g) &=\sup_{\| \nabla \varphi \|_\infty =1} \langle  g \star_{v} \psi_{\varepsilon}- g, \varphi\rangle  \\
    &=\sup_{\| \nabla \varphi \|_\infty =1} \sum_{\ell \in \Z} \int_{\T \times \R} \rho^{\ell}(x) \varphi(x,v) \big( \psi_{\varepsilon}(v-u^{\ell}(x))  - \delta_{v = u^{\ell}(x)} \big) \dd v \dd x\\
    & = \sup_{\| \nabla \varphi \|_\infty =1} \sum_{\ell \in \Z} \int_{\T} \rho^{\ell}(x) \int_{\R}  \varphi(x,v) \big( \psi_{\varepsilon}(v-u^{\ell}(x))  - \delta_{v = u^{\ell}(x)} \big) \dd v \dd x \\
    & = \sup_{\| \nabla \varphi \|_\infty =1} \sum_{\ell \in \Z} \int_{\T} \rho^{\ell}(x) \int_{\R}  \varphi(x,v + u^{\ell}(x)) \big( \psi_{\varepsilon}(v)  - \delta_{v = 0} \big) \dd v \dd x.
\end{align*}
For every $x \in \T$ and $\ell \in \Z,$ $\varphi(x,\cdot+u^{\ell}(x)) \in \textnormal{Lip}_{1}(\R)$, so we get

\begin{align*}
    \int_{\R}  \varphi(x,v + u^{\ell}(x)) \big( \psi_{\varepsilon}(v)  - \delta_{v = 0} \big) \dd v \leq \mathcal{W}_{1}( \psi_{\varepsilon}, \delta_{0}).
\end{align*}
Since $\rho^{\ell} \geq 0$ for every $\ell$ we eventually obtain
\begin{align*}
    \mathcal{W}_1(g \star_{v} \psi_{\varepsilon}, g) \leq \sum_{\ell \in \Z} \int_{\T} \rho^{\ell}(x) \dd x \mathcal{W}_{1}(\psi_{\varepsilon},\delta_{0}) \leq \varepsilon \int_{\T \times \R}   g(\dd x, \dd v) C_{\psi}
\end{align*}
where $C_{\psi} = \int_{\R} \psi(v) |v| \dd v.$
\end{proof}
\end{lemme}
\subsection{Main result}
Fix $T > 0$, $0 < \eta < 1$ and $f_0 \in \mathcal{A}.$ The approximate \textit{multi-fluid} solution is defined for every $N \in \N^{\star}$ and $0 < \delta < L$ by induction,
\begin{align} \label{def_remapping_procedure}
\begin{cases}
f^{N}_{\delta}(0) = G^{N} \# f_0 \\
f^{N}_{\delta}(s_{k+1}) = \mathcal{R}^{N} \varphi_{s_{k+1},s_{k}}^{\delta} \# f^{N}_{\delta}(s_{k}), \quad k \in \N,
\end{cases}
\end{align}
and between two consecutive remapping times $s_{k} < s_{k+1}$ by
\begin{align} \label{def_f_approx}
    f^{N}_{\delta}(t) = \varphi_{t,s_{k}}^{\delta} \# f^{N}_{\delta}(s_k), \quad t \in (s_{k},s_{k+1})
\end{align}
where $\varphi_{t, s_{k}}^{\delta}$ denotes the non linear flow associated with the regularized Vlasov--Poisson system $\textnormal{(VP)}_{\delta}$ and where the sequence of remapping times is given by
\begin{equation} \label{def_sk}
\begin{cases}
s_0 = 0,\\
s_{k+1} :=   \sup \left\{ t \geq s_k, \, \underset{ \ell \in \Z } \sup \displaystyle \int_{s_k }^{t} \| \partial_x u^{\ell}_{N,\delta}(s) \|_{L^{\infty}(\T)} \dd s \leq  \eta \right\},  \quad k \in \N.
\end{cases}
\end{equation}
It writes under the form
\begin{align}
\displaystyle f^{N}_{\delta}(t,x,v) = \sum_{\ell \in \Z} \rho^{\ell}_{N,\delta}(t,x) \o \delta_{v = u^{\ell}_{N,\delta}(t,x)}.
\end{align}
A key feature of the remapping operator is that, after each remapping, the approximate solution is restarted with constant fluid velocities. Consequently, as long as $f^N_\delta$ and the sequence $(s_k)_{k\in\mathbb{N}}$ are well-defined on $[0,T]$, Proposition \ref{lower_bound_local_existence_time}, together with the preservation of the total mass under the remapping procedure, yields the uniform lower bound
\begin{align}
    s_{k+1}-s_{k} \geq  \sqrt{\frac{ \eta(1-\eta) \delta}{ \| f_0 \|_{L^{1}(\T \times \R)} + L } }, \quad k \in \N.
\end{align} 
Hence, the final time $T$ is reached after at most $r_{T} \sim \frac{T}{\sqrt{\delta}}$ iterations.
Our main result shows, in particular, that the construction \eqref{def_sk}--\eqref{def_f_approx} is well-defined on the whole interval $[0,T]$. Moreover, it provides an error estimate between the solution $f$ of the Vlasov--Poisson system \textnormal{(VP)} and its \textit{multi-fluid} approximation $f_{\delta}^{N}$, with a convergence rate of order $\mathcal{O}(h^{2/3})$. Before stating the main result, we recall that if $f_0 \in \mathcal{A}$ then $\int_{\R} \| f_0(\cdot,v) \|_{L^{\infty}(\T)} \dd v < +\infty$ and
\begin{align}
  \sum_{\ell \in \Z} \Big \| \int_{v_{\ell-\frac{1}{2}}^{N}}^{v_{\ell+\frac{1}{2}}^{N}} f_0(\cdot,v) \dd v \Big \|_{L^{\infty}(\T)} \leq \int_{\R} \| f_0(\cdot,v) \|_{L^{\infty}(\T)} \dd v.
\end{align}
Therefore for every $f_0 \in \mathcal{A}$
\begin{align}
    \underset{ N \in \N^{\star} } \sup \: \sum_{\ell \in \Z} \Big \| \int_{v_{\ell-\frac{1}{2}}^{N}}^{v_{\ell+\frac{1}{2}}^{N}} f_0(\cdot,v) \dd v \Big \|_{L^{\infty}(\T)} \leq \int_{\R} \| f_0(\cdot,v) \|_{L^{\infty}(\T)} \dd v.
\end{align}
So initial data $f_0 \in \mathcal{A}$ are such that $G^{N} \# f_0$ fits into the assumptions of Theorem \ref{main_theorem},  uniformly with respect to the grid size $h$,  taking $u^\ell_0 = v^N_\ell$ and $\lambda^\ell= v^N_\ell$. Our main result is the following.
\begin{theoreme}[Convergence of the \textit{multi-fluid} approximation] Fix $T > 0$, $f_0 \in \mathcal{A}$  and $\eta \in (0,1)$.  For every $0 < \delta < L$ and $N \in \N^{\star}$, the approximate solution \eqref{def_remapping_procedure}-\eqref{def_sk}
is well-defined on $[0,T].$ The approximate solution $f^{N}_{\delta} :  [0, T] \longrightarrow K$ is piece-wise continuous in time with respect to the $\mathcal{W}_{1}(\T \times \R)$ topology with
\begin{align} \label{local_continuity_in_W1}
   \underset{t \rightarrow s_{k}^{+} } \lim f^{N}_{\delta}(t) =  f^{N}_{\delta}(s_{k}),  \quad \underset{t \rightarrow s_{k+1}^{-} } \lim f^{N}_{\delta}(t) = \varphi_{s_{k+1},s_{k}}^{\delta} \# f^{N}_{\delta}(s_k),\quad  k \in \N.
\end{align}
Furthermore, it satisfies the error estimate
\begin{align} 
 \!\!\! \underset{ t \in [0,T] } \sup \mathcal{W}_{1}(f(t), f^{N}_{\delta}(t))\leq  e^{a(T)} \Big(\mathcal{W}_{1}(G^{N}  \# f_0, f_{0}) + \delta\Big)+ \frac{h}{2}\|f_0\|_{L^{1}(\T \times \R)} \sum_{k = 0}^{r_{T}-1} e^{a(s_k)}, \label{main_error_estimate}
\end{align}
where the function $a$ is given in \eqref{function_a_t} and the final remapping iterate defined by
\begin{align}
   r_{T} := \inf \lbrace k \in \N \: : \: s_{k} \geq T \rbrace  \label{final_remapping_iterate}
\end{align}
satisfies
\begin{align}
    r_{T} \lesssim \frac{T}{\sqrt{\delta}}.
\end{align}
The error bound is minimal for $\delta \sim h^{\frac{2}{3}}.$ As a consequence, we have the convergence estimate
\begin{align}
    \underset{ t \in [0,T]} \sup \mathcal{W}_{1}\Big (f(t), f_{\delta}^{N}(t) \Big) = \mathcal{O}(h^{\frac{2}{3}}),\label{convergence_estimate_f}\\
    \underset{ t \in [0,T]} \sup \|E(t) - E^{N}_{\delta}(t) \|_{L^{1}(\T)} = \mathcal{O}(h^{\frac{2}{3}}). \label{convergence_estimate_E}
\end{align}
\end{theoreme}
This convergence estimate applies to weak solutions. In particular, no regularity assumption is imposed on $f$, beyond the boundedness and integrability required to define the solution. The proof is restricted to the one-dimensional setting, as it relies crucially on a weak-strong stability estimate together with an $L^{\infty}([0,T]\times\mathbb{T})$ bound on the electric field, both of which are specific to the one-dimensional framework.
Our error estimate \eqref{main_error_estimate} provides a link between two well-established paradigms in the theory of non-collisional plasmas by establishing the convergence of the multi-fluid approximation as $N\longrightarrow+\infty$ over arbitrarily large time intervals. This result contributes to the unified framework introduced in \cite{BEH}, which establishes a connection between fluid and kinetic descriptions of non-collisional plasmas through the so-called \textit{multi-phasic} formulation.

\section{Proof of the main result} 
\label{sec:proof_main_result}
Fix $T  > 0$, $f_0 \in \mathcal{A}$ and $\eta \in (0,1)$. 

\emph{\textbf{Step 1: Well-definiteness of the approximate solution.}}
For every $N \in \N^{\star}$ and $0 < \delta < L$, we prove by induction the following statement:

\begin{align*}
\forall k \in \N, \: f^{N}_{\delta} \textnormal{ is well-defined on } [0,s_{k+1}].
\end{align*}

As for the base case, we have $f^{N}_{\delta}(0) = G^{N} \# f_0 \in K$. Since $f_0 \in \mathcal{A}$, $f^{N}_{\delta}(0)$ verifies the assumptions of Theorem \ref{main_theorem} with the regularity index $m= 0$ (taking $\lambda^\ell= v^N_\ell$). Thus, $f^{N}_{\delta} : [0,s_{1}) \longrightarrow K$ is well defined and continuous on $[0,s_{1})$ for the $\mathcal{W}_{1}(\T \times \R)$ topology. Moreover, by continuity in time, it has a limit as $t \longrightarrow s_{1}^{-}$. Using the Lagrangian form of the momentum equation, we have 
\begin{align}
   \underset{\ell \in \Z} \sup \|u_{N,\delta}^{\ell}(s_{1}^{-}) - v_{\ell}^{N} \|_{L^{\infty}(\T)} \leq s_1 (\|f_0\|_{L^{1}(\T \times \R)} + L).
\end{align}
It yields thanks to the remapping estimate \eqref{estimate_sum_linf_after_remap} and the continuity equation,
\begin{align}
    \sum_{\ell \in \Z} \| \rho^{\ell}_{N,\delta}(s_1) \|_{L^{\infty}(\T)} \leq \Big( 1 + \frac{2 s_1(\|f_0\|_{L^{1}(\T \times \R)} +  L)}{h} \Big) e^{\eta} \sum_{\ell \in \Z} \| \rho^{\ell}_{N,\delta}(s_0) \|_{L^{\infty}(\T)}.
\end{align}
Thus, $f^{N}_{\delta}(s_1)$ still verifies the assumption of Theorem \ref{main_theorem} (again taking $m=0$ and $\lambda^\ell= v^N_\ell$). So $f^{N}_{\delta}$ is well-defined on $[0,s_1].$
By immediate induction, we deduce that $f^{N}_{\delta}$ is well-defined on $[0,s_{k+1}]$ for every $k \in \N$. 
Moreover, we have
\begin{align}
    \sum_{\ell \in \Z} \| \rho^{\ell}_{N,\delta}(s_k) \|_{L^{\infty}(\T)} \leq \Big( 1 + \frac{2 s_{k}(\|f_0\|_{L^{1}(\T \times \R)} +  L)}{kh} \Big)^{k} e^{k\eta} \sum_{\ell \in \Z} \| \rho^{\ell}_{N,\delta}(s_0) \|_{L^{\infty}(\T)}.
\end{align}
Eventually observe that the sequence of remapping times satisfies
$s_{k} \gtrsim k \sqrt{\delta}. $ Thus, $\limsup s_k = +\infty$ and the final remapping time satisfies
\begin{align}
    r_{T} \lesssim \frac{T}{\sqrt{\delta}}.
\end{align}

\emph{\textbf{Step 2: Error estimate}.} 
For $k = 0,\dots,r_{T}$, we decompose the error between two consecutive remapping time $s_{k} \leq s_{k+1}$ using a triangular inequality as,
\begin{align}\label{triangle_inequality_error}
\mathcal{W}_{1}\Big(f^{N}_{\delta}( s_{k+1}), f_{\delta}(s_{k+1}) \Big)
&\leq \mathcal{W}_{1}\big(\mathcal{R}^{N} f^{N}_{\delta}(s_{k+1}^{-}), f^{N}_{\delta}(s_{k+1}^{-}) \big)
\\
&+ \mathcal{W}_{1}\big(  f^{N}_{\delta}(s_{k+1}^{-}), f_{\delta}(s_{k+1}) \big),\nonumber
\end{align}
where $f^{N}_{\delta}(s_{k+1}) = \mathcal{R}^{N} f^{N}_{\delta}(s_{k+1}^{-}).$
Thanks to \eqref{w1_stab} we have for the remapping error,
\begin{align*}
\mathcal{W}_{1}\big(\mathcal{R}^{N} f^{N}_{\delta}(s_{k+1}^{-}), f^{N}_{\delta}(s_{k+1}^{-}) \big) \leq \frac{h}{2}  \int_{\T \times \R} f_{\delta}^{N}(s_{k+1}^{-})(\dd x,\dd v).
\end{align*}
Thanks to the conservation of the total mass \eqref{mass_preserved}, \eqref{mass_preserving} we readily get by induction that
\begin{align*}
\mathcal{W}_{1}\big(\mathcal{R}^{N} f^{N}_{\delta}(s_{k+1}^{-}), f^{N}_{\delta}(s_{k+1}^{-}) \big) \leq \frac{h}{2}  \| f_0 \|_{L^{1}(\T \times \R)}.
\end{align*}
As for the consistency error, using the stability estimate \eqref{eq:multi_fluid_error} we get
\begin{align}
    \mathcal{W}_{1}\big(  f^{N}_{\delta}(s_{k+1}^{-}), f_{\delta}(s_{k+1}) \big) \leq e^{a_{k}(s_{k+1})} \mathcal{W}_{1}\big( f_{\delta}^{N}(s_k),f_{\delta}(s_k) \big)
\end{align}
where $ a_k(t):= 8 \int_{s_k}^{t} (1 +  \underset{\delta > 0} \sup \:  \| \rho_{f_\delta}(s) \|_{L^{\infty}(\T)}) \dd s.$ 
Eventually, we glean 
\begin{align} \label{recursive_error}
    &\mathcal{W}_{1}\Big(f^{N}_{\delta}(s_{k+1}), f_{\delta}(s_{k+1}) \Big) \leq  e^{a_{k}(s_{k+1})} \mathcal{W}_{1}(f_{\delta}^{N}(s_k),f_{\delta}(s_k))+ \frac{h}{2}\|f_0\|_{L^{1}(\T \times \R)}, \quad k \in \N.
\end{align}
For $t \in [0,T]$ there exists $k_{t} \in \N$ such that $t \in [s_{k_t},s_{k_{t}+1}).$ So we get
\begin{align}
    \mathcal{W}_{1}(f^{N}_{\delta}(t), f_{\delta}(t)) \leq e^{a_{k_{t}}(t)} \mathcal{W}_{1}(f_{\delta}^{N}(s_{k_{t}}), f_{\delta}(s_{k_{t}}))
\end{align}
and by iteration of \eqref{recursive_error} we get
\begin{align}
   \mathcal{W}_{1}(f^{N}_{\delta}(t), f_{\delta}(t)) \leq e^{a(t)} \mathcal{W}_{1}(f^{N}_{\delta}(0), f_{0}) + \frac{h}{2}\|f_0\|_{L^{1}(\T \times \R)} \sum_{k = 0}^{k_{t}} e^{a(s_k)}. 
\end{align}
Since the function $t \longmapsto a(t)$ is increasing we infer
\begin{align}
    \underset{ t \in [0,T] } \sup \mathcal{W}_{1}(f^{N}_{\delta}(t), f_{\delta}(t))\leq  e^{a(T)} \mathcal{W}_{1}(f^{N}_{\delta}(0), f_{0}) + \frac{h}{2}\|f_0\|_{L^{1}(\T \times \R)} \sum_{k = 0}^{r_{T}-1} e^{a(s_k)}.
\end{align}

\emph{\textbf{Step 3: Convergence estimate}.} 
To conclude, we again use a weak-strong stability estimate due to Hauray (see estimate 2.19 in \cite{Hauray}). For every $t \in [0,T]$, since $f_{\delta}(0) = f(0) = f_0,$ we have
\begin{align}
    \mathcal{W}_{1}(f(t),f_{\delta}(t)) \leq \delta e^{a(t)}.
\end{align}
So by a triangular inequality
\begin{align*}
  \underset{ t \in [0,T] } \sup \mathcal{W}_{1}(f(t),f_{\delta}^{N}(t))  &\leq e^{a(T)} \big( \mathcal{W}_{1}(G^{N} \# f_0, f_{0}) + \delta \big) + \frac{h}{2}\|f_0\|_{L^{1}(\T \times \R)} \sum_{k = 0}^{r_{T}-1} e^{a(s_k)}\\
  &\leq \frac{ e^{a(T)} \| f_0 \|_{L^{1}(\T \times \R)}}{2} \left(h + \frac{2\delta}{\| f_0 \|_{L^{1}(\T \times \R)} }+ hr_{T}\right),
\end{align*}
By duality, we infer a control of the error at the level of the electric field:
\begin{align}
  \underset{ t \in [0,T]} \sup \| E(t)- E^{N}_{\delta}(t)\|_{L^1(\T)} \leq \underset{ t \in [0,T] } \sup  \mathcal{W}_{1}(f(t),f^{N}_{\delta}(t)).
\end{align}
Since $r_T \lesssim \frac{T}{\sqrt{\delta}}$, the error bound is a function of $\delta$ of the form
\begin{align*}
    e(\delta) := h + \frac{2\delta}{\| f_0 \|_{L^{1}(\T \times \R)} }+ \frac{C h T}{\sqrt{\delta} }, \quad \delta > 0, 
\end{align*}
where $C$ is a constant that depends only on $\eta$ and on the initial data $f_0.$
It is minimal for $\delta \sim h^{\frac{2}{3}}$. Thus, we get the expected convergence rate.
\section{A numerical experiment} \label{sec:numerical_results}
This section is devoted to assessing the efficiency of our approach on the two-stream instability, a challenging test case for fluid-based methods \cite{besse-bag,crestetto-bag}. We combine our approach with a numerical discretization of the fluid equations based on a semi-Lagrangian method with cubic splines in space and a Strang splitting scheme in time. The initial data are discretized on a phase-space grid $(x_i,v_{\ell}) \subset [0,L) \times \R$ where 
\begin{align}
    &x_i = \frac{iL}{N_x}, \quad i = 0,\dots ,N_x -1\\
    &v_{\ell} = A + \ell h , \quad \ell =0 \dots  N_v.
\end{align}
Here, $N_x, N_v \in \N^{\star}$, $A,B \in \R$, $L> 0$ and $h = (B-A)/N_v$ denote the numerical parameters. The initial fluid densities are approximated at the grid points according to
\begin{align}
  &\rho^{\ell}_0(x_i) \approx  f_0(x_i,v_{\ell}) h, \quad i = 0,\dots, N_x -1, \: \ell = 0,\dots, N_v, 
\end{align}
where $f_0$ is chosen to be continuous here.
The particular choice of numerical discretization is not essential to our approach, and other numerical methods could equally be used to solve the fluid equations. We consider the two-stream instability with initial data of the form
\begin{align}
   &f_{0}(x,v) =  \frac{1}{2}( \mathcal{M}(v-u_0)  + \mathcal{M}(v+u_0))(1+ \epsilon \cos(kx) ), \quad (x,v) \in [0,L) \times \R,
\end{align}
where $\epsilon =10^{-3}$ is the size of the perturbation, $k=0.2$, $u_0=2.4$ and $\mathcal{M}$ is the centered \textit{Maxwellian} given by
\begin{equation}
    \mathcal{M}(v) = \frac{e^{-v^2/2}}{\sqrt{2 \pi}},  \quad v \in \R.
\end{equation}
We set $A = -8, B = 8$ and $L = 2\pi/k.$
The final time is $T = 50$, with a time step $\Delta t = 0.1$. We choose the remapping parameter $\eta=0.5$. In the numerical implementation, the remapping criterion is discretized from the continuous criterion \eqref{def_sk} using a simple rectangle rule for the time integral and finite differences to approximate the Lipschitz norms of the fluid velocities. For the numerical illustrations, we also incorporate a smoothing step by convolution into the remapping operator, using a cubic-spline kernel with support of size $h$. This modification does not affect the analysis, as shown in Lemma \ref{convol_lemma}, and is introduced solely to produce smoother visualizations of the distribution function. We compare our method, referred to as 'SL', with a standard semi-Lagrangian scheme applied directly to the Vlasov--Poisson system, denoted by 'ref' in the following figures. Both methods are computed on the same phase-space grid. During the simulation, the remapping operator is applied $16$ times. The implementation is not yet optimized, and a detailed study of the computational cost of the proposed approach compared with more standard methods, in particular in higher dimensions, is left for future work \cite{multi-fluid-num-study}.
Regarding the regularization parameter $\delta$ in the Coulomb kernel, the Poisson equation is solved numerically using a discrete Fourier transform. Thus, the effective regularization is expected to be at most of the order of the velocity-grid size $h$.
In Figure \ref{fig:elec_TSI}, we plot the time evolution of the electric energy  obtained by the two methods (in semi-$\log$ scale). In particular, we recover the instability rate $\gamma=0.2258$ predicted by the linear theory. We also plot the contours of the distribution function obtained by our method and the semi-Lagrangian scheme at the final time $T=50$.  Finally, in Figure \ref{fig:u_TSI}, we represent the graphs of few selected velocities before the appearance of a shock and after the remapping operator is applied.  As expected, after the remapping, each fluid velocity becomes constant and the simulation continues. We plot the fluid velocities corresponding to the discrete velocities between $- 0.041$ and $0.029$.  

\begin{figure}[ht!]
    \centering
    \includegraphics[width=8cm,height=6cm]{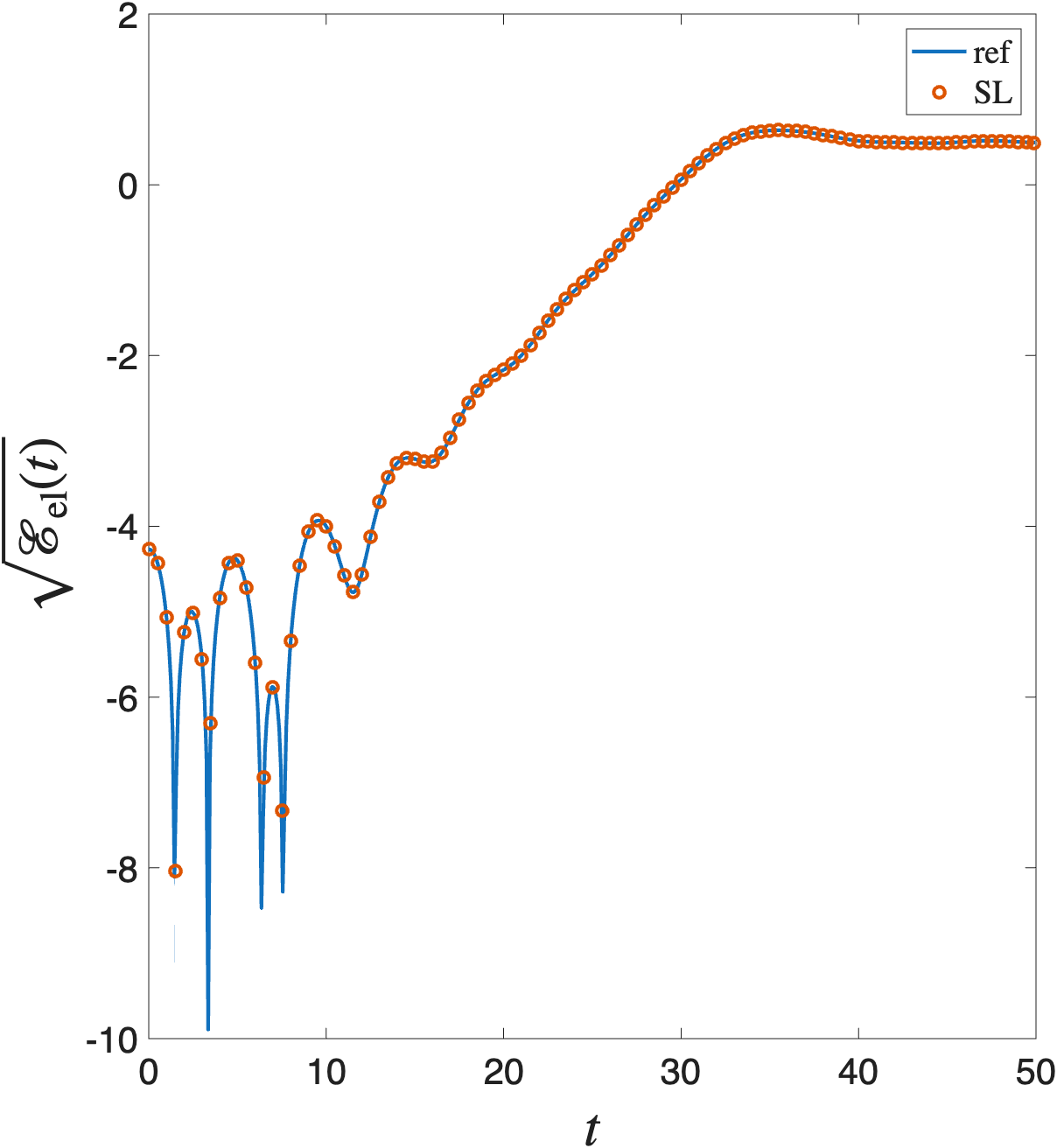}
 \caption{Two stream instability test: time evolution (semi-$\log$ scale) of the (square root of the) electric energy for the new method 'SL'  and the reference method 'ref'.}
 \label{fig:elec_TSI} 
\end{figure}

\begin{figure}[ht!]
    \centering
    \includegraphics[width=7cm, height=6cm]{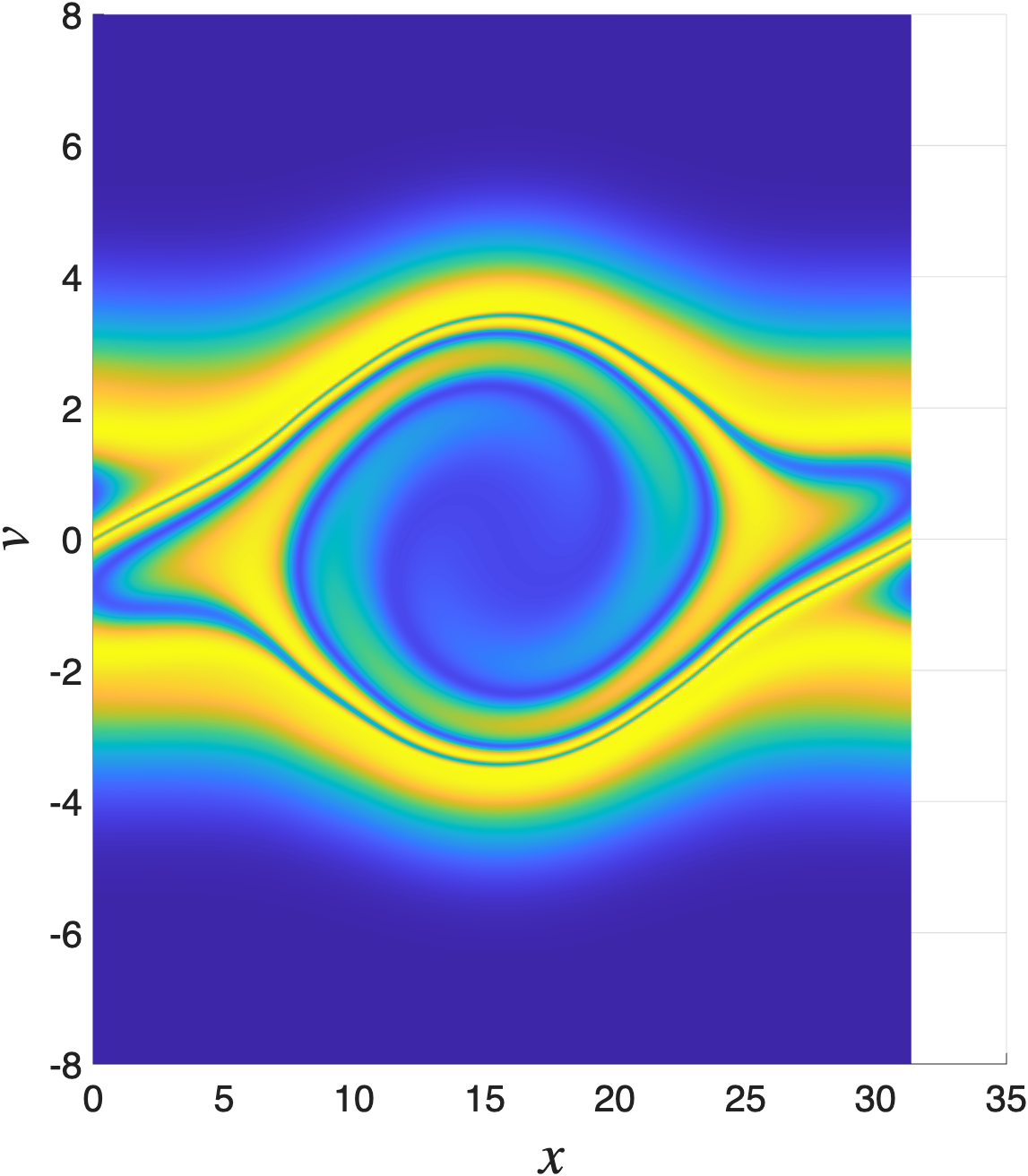}
    \includegraphics[width=7cm, height=6cm]{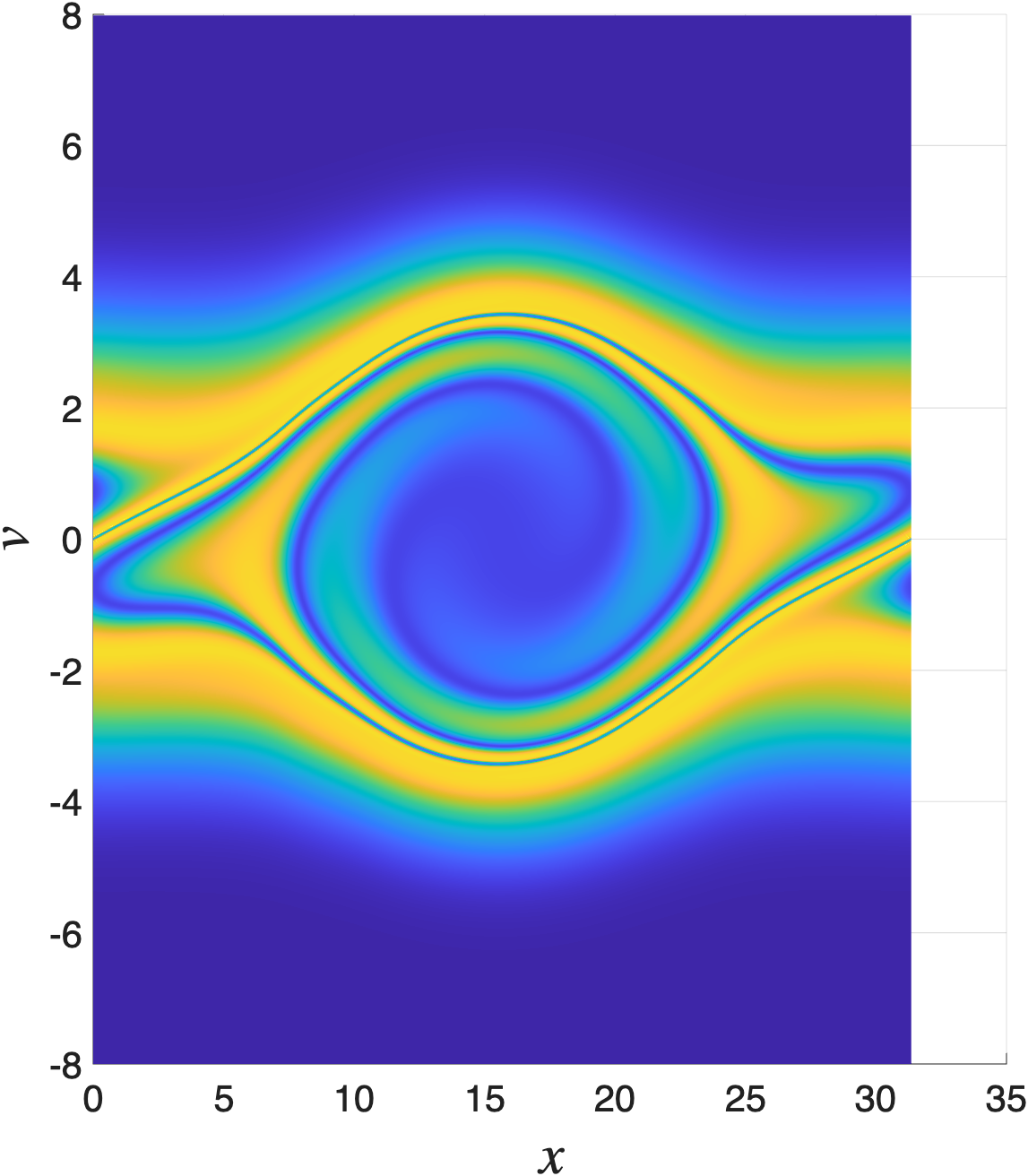}
    \caption{Two stream instability test: Contours of the distribution function at time $T = 50$, for the new method with a mesh  $512$x$1024$ (left) and the reference method with a mesh $512$x$1024$  (right).}\label{fig:f_TSI} 
\end{figure}
\begin{figure}[ht!]
    \centering
    \includegraphics[width=0.45\linewidth]{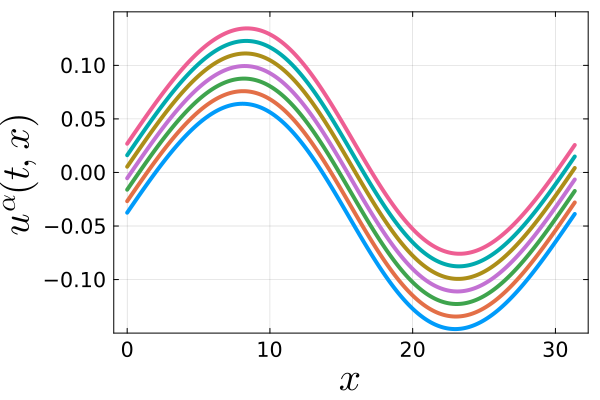}
    \includegraphics[width=0.45\linewidth]{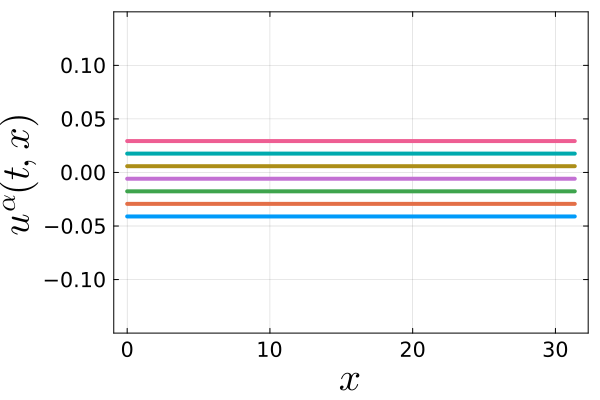}
    \caption{Two stream instability test: Graphs of few velocities before (left) and after (right) the first remapping, for few values of $\ell$.}\label{fig:u_TSI} 
\end{figure}
\newpage
\section{Appendix}
Let $f_0 \in \mathcal{A}$ and $(f,E)$ be the mild solution of the Vlasov--Poisson system given in Theorem 7.4 \cite{bostan-vp-1d}. We shall prove
\begin{align*}
    f \in \mathscr{C}\Big ( [0,T]; L^{1}(\T \times \R) \Big).
\end{align*}
It readily implies the following regularity for the electric field
\begin{align*}
    E \in \mathscr{C}\Big ( [0,T]; L^{1}(\T) \Big).
\end{align*}
Recall that if $f_0$ is a given measurable function then for every $\hat{f_0}$ in the class of $f_0$ there exists a set $\mathcal{N}_{0}$ of zero Lebesgue measure such that $\hat{f_0}$ is defined on $\mathcal{D}_0 := (\T \times \R) \setminus \mathcal{N}_{0}$ and is equal to $f_0.$ Since the flow of the Vlasov--Poisson system is measure preserving for every $t \in [0,T]$ the set $\varphi_{t,0}(\mathcal{N}_{0}) $ is also of zero measure. Therefore, for every $t \in [0,T],$ $f(t) = f_0 \circ \varphi_{0,t}$ is well-defined on $ \mathcal{D}_{t}:=(\T\times \R)\setminus \varphi_{t,0}(\mathcal{N}_{0})$ which is a set of full measure. Fix $s \in [0,T]$ and $(t_{n})_{n \in \N} \subset [0,T]$ such that $t_{n} \longrightarrow s.$ We shall prove that for every $\varepsilon > 0$ there exists $N(\varepsilon, s) \in \N$ such that for every $n \in \N,$
\begin{align}
    n \geq N(\varepsilon, s) \Longrightarrow \| f(s_n) - f(s) \|_{L^{1}(\T \times \R)} < \varepsilon.
\end{align}
We proceed by regularization by convolution and truncation in velocity of the initial data. Since $f_0 \in L^{1} (\T \times \R)$ there exists a sequence $(f_{0}^{k})_{k \in \N} \subset \mathscr{C}_{c}^{\infty}( \T \times \R)$ such that
\begin{align}
    \| f_0^{k} - f_0 \|_{L^1(\T \times \R)} \longrightarrow 0,\\
    \forall k \in \N, \quad \textnormal{ supp } f_0^{k} \subset \T \times [-k,k].
\end{align}
Observe now that for every $n \in \N$ and $k \in \N$, we have the decomposition
\begin{align}
     &f(s_n) - f(s)  = (f_0^{k} \circ \varphi_{0,s_n} - f_0^{k} \circ \varphi_{0,s})\\
     &+(f_0 \circ \varphi_{0,s_n} - f_0^{k} \circ \varphi_{0,s_n}) +(f_0^{k} \circ \varphi_{0,s} - f_0 \circ \varphi_{0,s}).
\end{align}
By triangular inequality we infer
\begin{align*}
    \| f(s_n) - f(s) \|_{L^1} \leq &\| f_0 \circ \varphi_{0,s_n} - f_0^{k} \circ \varphi_{0,s_n} \|_{L^1} \nonumber\\
    &+  \| f_0^{k} \circ \varphi_{0,s} - f_0 \circ \varphi_{0,s} \|_{L^{1}} + \| f_0^{k} \circ \varphi_{0,s_n} - f_0^{k} \circ \varphi_{0,s} \|_{L^1}.\nonumber
\end{align*}
Recall that the flow of the Vlasov--Poisson preserves the $L^{1}$-norm so that
\begin{align}
    \| f_0 \circ \varphi_{0,s_n} - f_0^{k} \circ \varphi_{0,s_n} \|_{L^1} = \| f_0 - f_0^{k} \|_{L^{1}},\\
    \| f_0 \circ \varphi_{0,s} - f_0^{k} \circ \varphi_{0,s} \|_{L^1} = \| f_0 - f_0^{k} \|_{L^{1}}.
\end{align}
Fix $\varepsilon > 0$. Then there exists $N_{1}(\varepsilon) \in \N$ such that for every $k \geq N_{1}(\varepsilon)$, $\| f_0 - f_0^{k} \|_{L^1} < \frac{\varepsilon}{3}.$ So now fix $r \geq N_{1}(\varepsilon).$ We then get for every $n \in \N,$
\begin{align}
    \| f(s_n) - f(s) \|_{L^1}  \leq \frac{2 \varepsilon}{3} + \| f_0^{r} \circ \varphi_{0,s_n} - f_0^{r} \circ \varphi_{0,s} \|_{L^{1}}.
\end{align}
As $f_{0}^{r} \circ \varphi_{0,s_n}$ is a continuous function on $[0,T] \times \T \times \R,$ we have
\begin{align}
   \forall (x,v) \in \T \times \R, \quad f_0^{r} \circ \varphi_{0,s_n}(x,v) \longrightarrow f_0^{r} \circ \varphi_{0,s}(x,v).
\end{align}
To conclude by dominated convergence, it remains to show that  $\textnormal{ supp }f_{0}^{r} \circ \varphi_{0,s_n}$ is embedded in a compact set that is independent of $n$. In this regard, remark that for every $(x,v) \in \T \times \R$ we have for $n \in \N$ large enough,
$$
|v - V(0,s_n,x,v)| \leq s_n \| E \|_{L^{\infty}([0,T] \times \T)} \leq T \| E \|_{L^{\infty}([0,T] \times \T)}.
$$
It yields for $n$ large enough
$$
\textnormal{ supp }f_{0}^{r} \circ \varphi_{0,s_n} \subset \T \times [-(r+ T \| E \|_{L^{\infty}([0,T] \times \T)}), r+T \| E \|_{L^{\infty}([0,T] \times \T)} ].
$$
So the Lebesgue dominated convergence applies. We deduce the existence of $N_2(\varepsilon,s,r) \in \N$ such that for $n \geq N_{2}(\varepsilon,s,r),$
$\| f_0^{r} \circ \varphi_{0,s_n} - f_0^{r} \circ \varphi_{0,s} \|_{L^{1}} < \frac{\varepsilon}{3}$. Thus, for $n \geq \max(N_{1}(\varepsilon);N_{2}(\varepsilon,s,r)))$, $\| f(s_n) - f(s) \|_{L^1} < \varepsilon.$


\bibliographystyle{plain}
\bibliography{biblio}

\end{document}